\documentclass[11pt,reqno]{amsart}
\usepackage[utf8]{inputenc}
\usepackage[english]{babel}
\usepackage{xcolor}
\usepackage{tikz}
\usepackage{amsmath,amssymb,amsthm,mathtools}
\usepackage[a4paper,margin=2.6cm]{geometry}
\usepackage[colorlinks=true,linkcolor=blue!50!black,citecolor=blue!50!black,urlcolor=blue!50!black]{hyperref}

\numberwithin{equation}{section}
\newtheorem{thm}{Theorem}[section]
\newtheorem{lem}[thm]{Lemma}

\newtheorem{cor}[thm]{Corollary}
\theoremstyle{definition}
\newtheorem{defn}[thm]{Definition}
\newtheorem{ex}[thm]{Example}
\newtheorem{prob}{Problem}
\theoremstyle{remark}
\newtheorem{obs}[thm]{Remark}

\newcommand{\C}{\mathbb{C}}
\newcommand{\R}{\mathbb{R}}
\newcommand{\D}{\mathbb{D}}
\newcommand{\B}{\mathbb{B}^n}
\newcommand{\Pn}{\mathbb{D}^n}
\newcommand{\Fa}{\mathcal{F}_\alpha}
\newcommand{\Aut}{\operatorname{Aut}}
\newcommand{\ord}{\operatorname{ord}}
\newcommand{\tr}{\operatorname{tr}}
\newcommand{\Hess}{\operatorname{Hess}}
\newcommand{\Id}{I}
\newcommand{\RE}{\operatorname{Re}}
\newcommand{\diam}{\operatorname{diam}}
\newcommand{\pd}[2]{\frac{\partial #1}{\partial #2}}
\newcommand{\pdd}[3]{\frac{\partial^2 #1}{\partial #2\,\partial #3}}
\newcommand{\bil}[2]{\langle #1,#2\rangle}

\makeatletter
\def\tocsection#1#2#3{\indentlabel{\@ifnotempty{#2}{\makebox[2.1em][l]{\ignorespaces#1 #2.}}}#3}
\makeatother

\begin{document}

\title[The Schwarzian derivative in several variables]{The Schwarzian derivative\\ in several complex variables\\ univalence and geometry}
\author{Rodrigo Hernández}
\makeatletter
\g@addto@macro\@setauthors{%
  {\centering\footnotesize
   Facultad de Ingeniería y Ciencias, Universidad Adolfo Ibáñez\par
   Av. Padre Hurtado 750, Viña del Mar, Chile\par
   \texttt{rodrigo.hernandez@uai.cl}\par
   ORCID \href{https://orcid.org/0000-0002-2787-4598}{0000-0002-2787-4598}\par}%
  \vskip 6pt}
\makeatother

\begin{abstract}
In one complex variable the Schwarzian derivative characterizes the Möbius transformations, arises from a second-order linear equation, and controls univalence through the theorems of Nehari and of Ahlfors--Weill. This chapter presents what is known about these three properties for locally biholomorphic mappings in $\mathbb{C}^n$. Starting from Oda's definition we construct a bilinear operator whose norm, taken with respect to the Bergman metric, is invariant under the automorphisms of the ball, and we show that this operator measures the extrinsic curvature of the preimages of hyperplanes. We study the linearly invariant family defined by a bound on this norm, we prove that a sufficiently small bound implies univalence, and we describe the extension of Ahlfors--Weill type, which takes values in the space of hyperplanes. The polydisk and convex mappings are also treated, and open problems and exercises are included.
\end{abstract}


\maketitle
\tableofcontents

\section{The Schwarzian derivative in one variable}\label{sec:modelo}

The Schwarzian derivative of an analytic and locally univalent function $f$ is
\begin{equation}\label{eq:schwarzclasica}
Sf=\Bigl(\frac{f''}{f'}\Bigr)'-\frac12\Bigl(\frac{f''}{f'}\Bigr)^2 .
\end{equation}
The expression appears already in the nineteenth century, in the work of Kummer and of Schwarz on the hypergeometric equation, and it has occupied a central place in geometric function theory ever since. Its usefulness rests on three properties, which are worth recalling in some detail because this whole chapter consists in asking what happens to them when the disk is replaced by a domain in $\C^n$.

The first is invariance. We have $Sf\equiv0$ if and only if $f$ is a Möbius transformation, and the chain rule
\[
S(f\circ g)=\bigl((Sf)\circ g\bigr)(g')^2+Sg
\]
shows that $S(T\circ f)=Sf$ for every Möbius transformation $T$. The Schwarzian therefore measures how far $f$ is from the Möbius transformations. If $f$ is defined in the unit disk $\D$, the quantity
\begin{equation}\label{eq:normaclasica}
\|Sf\|=\sup_{z\in\D}\,(1-|z|^2)^2\,|Sf(z)|
\end{equation}
is unchanged when $f$ is composed with a disk automorphism, because $(1-|z|^2)|\sigma'(z)|=1-|\sigma(z)|^2$ for every automorphism $\sigma$.

The invariance is verified by a calculation, but it is natural to ask how one arrives at the expression \eqref{eq:schwarzclasica}. The argument, which is essentially that of Schwarz, is an elimination of parameters. If $g=T\circ f$ with $T(w)=(aw+b)/(cw+d)$, then $(cf+d)\,g=af+b$, and three successive differentiations give
\[
c\,(fg)^{(k)}+d\,g^{(k)}-a\,f^{(k)}=0,\qquad k=1,2,3,
\]
a homogeneous linear system with the nontrivial solution $(c,d,-a)$. Its determinant vanishes, and expanding it and dividing by $2f'^2g'^2$ we obtain precisely $Sg=Sf$. Three differentiations are needed because there are three parameters to eliminate.

It is worth pausing at this point, because the order of the operator is no accident. The Möbius transformations $T(z)=(az+b)/(cz+d)$ form a group of complex dimension three, and each one is determined by the three numbers $T(z_0)$, $T'(z_0)$, $T''(z_0)$. On the other hand, the solutions of a differential equation of order $k$, solved for the highest derivative, form a family with $k$ parameters, the initial values $f(z_0),\dots,f^{(k-1)}(z_0)$. A differential operator that vanishes exactly on the Möbius transformations must therefore be of order three. One of order two vanishes on a family with only two parameters; this is what happens with the pre-Schwarzian $f''/f'$, whose kernel consists of the affine mappings. In other words, given a locally univalent $f$ and a point $z_0$, there is a unique Möbius transformation $M$ that agrees with $f$ to second order at $z_0$, and the Schwarzian measures the discrepancy at the third,
\begin{equation}\label{eq:osculante}
f(z)-M(z)=\tfrac16\,f'(z_0)\,Sf(z_0)\,(z-z_0)^3+O\bigl((z-z_0)^4\bigr).
\end{equation}
Three degrees of freedom, an operator of order three. This count, trivial in one variable, no longer balances in $\C^n$, and that is the first difficulty we shall encounter.

The second property ties the Schwarzian to a linear equation. If $Sf=2p$, the functions
\[
u_0=(f')^{-1/2},\qquad u_1=f\,(f')^{-1/2}
\]
are linearly independent solutions of
\begin{equation}\label{eq:lineal1}
u''+p\,u=0,
\end{equation}
and $f=u_1/u_0$. Conversely, the quotient of two independent solutions of \eqref{eq:lineal1} has Schwarzian $2p$. The consequence that interests us is the following.

\begin{lem}\label{lem:ceros1}
Let $f$ be analytic and locally univalent in a simply connected domain $\Omega\subset\C$, with $Sf=2p$. Then $f$ is univalent in $\Omega$ if and only if no nontrivial solution of \eqref{eq:lineal1} vanishes at two distinct points of $\Omega$.
\end{lem}

\begin{proof}
Every solution has the form $u=u_0(c_0+c_1f)$, and $u_0$ does not vanish. The solutions that vanish at $a$ are the multiples of $u_0\,(f-f(a))$, and one of them vanishes also at $b$ exactly when $f(b)=f(a)$.
\end{proof}

The third property is the one that gives geometric content to the previous two. The size of $\|Sf\|$ controls univalence. The fundamental result is due to Nehari \cite{Nehari}.

\begin{thm}[Nehari]\label{thm:nehari}
If $f$ is analytic and locally univalent in $\D$ and $\|Sf\|\le2$, then $f$ is univalent in $\D$.
\end{thm}

\begin{proof}
Suppose that $f(a)=f(b)$ with $a\ne b$. Let $\sigma$ be an automorphism of the disk that carries the real diameter onto the hyperbolic geodesic through $a$ and $b$. By the invariance of the norm, $g=f\circ\sigma$ satisfies $\|Sg\|\le2$ and takes the same value at two points of the interval $(-1,1)$. By Lemma~\ref{lem:ceros1} there is a nontrivial solution $u$ of $u''+pu=0$, with $2p=Sg$, that vanishes at those two points; since its zeros are isolated, we may take two consecutive zeros $x_1<x_2$ on the real axis.

On $(x_1,x_2)$ the function $v=|u|$ is positive and smooth, and a calculation with the Cauchy--Schwarz inequality gives $v''\ge\RE(\overline u\,u'')/|u|$. Therefore
\[
v''\ \ge\ -|p|\,v\ \ge\ -\frac{v}{(1-x^2)^2}.
\]
We compare with $y(x)=\sqrt{1-x^2}$, which is positive on $(-1,1)$ and satisfies $y''=-y/(1-x^2)^2$. The function $W=v'y-vy'$ has derivative $W'=v''y-vy''\ge0$. But the zeros of $u$ are simple, so that $W(x_1^+)=|u'(x_1)|\,y(x_1)>0$ and $W(x_2^-)=-|u'(x_2)|\,y(x_2)<0$. A nondecreasing function cannot pass from a positive value to a negative one, and this contradiction proves the theorem.
\end{proof}

The constant $2$ is sharp, and it is remarkable that the class so defined is not small, since it contains, for example, all convex functions. In the opposite direction, Kraus and Nehari proved that every univalent function in $\D$ satisfies $\|Sf\|\le6$. There is also a criterion in terms of the pre-Schwarzian $Pf=f''/f'$, due to Becker \cite{Becker}. If $(1-|z|^2)|zPf(z)|\le1$ in $\D$, then $f$ is univalent. For all of this we refer to \cite{Dur83}.

When Nehari's bound holds with room to spare, much more can be said. Ahlfors and Weill \cite{AW} showed that if $\|Sf\|\le2t$ with $t<1$, then $f$ admits a quasiconformal extension to the Riemann sphere, and they gave the formula
\begin{equation}\label{eq:AW1}
E_f(\zeta)=f(\zeta)+\frac{(1-|\zeta|^2)f'(\zeta)}{\overline{\zeta}-\tfrac12(1-|\zeta|^2)Pf(\zeta)},\qquad \zeta\in\D,
\end{equation}
for the value of the extension at the reflected point $1/\overline\zeta$. Geometrically, $E_f(\zeta)$ is the value at $1/\overline\zeta$ of the Möbius transformation that best approximates $f$ at $\zeta$. At $\zeta=0$, and for $f(z)=z+a_2z^2+\cdots$, we obtain $E_f(0)=-1/a_2$. This point has an interpretation in terms of the equation \eqref{eq:lineal1}. The solution normalized by $u(0)=1$, $u'(0)=0$ is $u=(f')^{-1/2}(1+a_2f)$, and it vanishes precisely where $f=-1/a_2$. To say that the Ahlfors--Weill extension sends the exterior of the disk to the complement of $f(\D)$ is to say, at the origin, that \emph{the normalized solution does not vanish}. We shall see that this formulation, rather than the formula \eqref{eq:AW1}, is the one that survives in higher dimension.

Finally, the bounds on $\|Sf\|$ are well organized in the language of the linearly invariant families of Pommerenke \cite{Pom}. A family $\mathcal F$ of functions $f(z)=z+a_2z^2+\cdots$ locally univalent in $\D$ is linearly invariant if with each $f$ it contains its Koebe transforms
\[
\frac{f\bigl(\frac{z+z_0}{1+\overline{z_0}z}\bigr)-f(z_0)}{(1-|z_0|^2)f'(z_0)},\qquad z_0\in\D .
\]
The order $\sup_{\mathcal F}|a_2|$ of the family governs the growth, the distortion and the covering properties of its members, and Pommerenke calculated, by a variational argument, that the order of the family $\{\|Sf\|\le\alpha\}$ is exactly $\sqrt{1+\alpha/2}$.

\medskip

\noindent\textbf{The problem.} In $\C^n$ the conformal mappings of the plane are replaced by biholomorphic mappings, and we ask how much of the above remains true. The precise questions are the following.
\begin{enumerate}
\item[(A)] Is there a differential operator that vanishes exactly on the Möbius transformations of $\C^n$ and that arises from a linear equation whose solutions recover the mapping as a quotient?
\item[(B)] Does this operator have an invariant norm, and does a bound on that norm imply univalence?
\item[(C)] What remains of the Ahlfors--Weill extension and of the geometry of the image?
\end{enumerate}
The answer to (A) is affirmative and is due to Oda \cite{Oda}. The equation \eqref{eq:lineal1} becomes an overdetermined system of partial differential equations, and the Schwarzian becomes a family of coefficients $S^k_{ij}F$. Section~\ref{sec:schwarziana} develops this construction with complete proofs, gathers the coefficients into a bilinear operator $S_F$, and gives it a geometric meaning, since $S_F$ is the curvature of the preimages of hyperplanes. The answer to (B) depends on the domain. In the ball $\B$ the Bergman metric gives a norm $\|S_F\|$ invariant under automorphisms, the mappings with $\|S_F\|\le\alpha$ form a compact linearly invariant family (Section~\ref{sec:familia}), and a sufficiently small bound implies univalence (Sections~\ref{sec:univalencia} and~\ref{sec:AW}). The three steps of Nehari's proof reappear, but each one costs more. The lemma on zeros now speaks of common zeros of $n$ solutions, the reduction to a diameter leaves a first-order term that does not exist in one variable, and the Sturm comparison must be replaced, in order to reach the definitive result, by the convexity of a density with respect to the Bergman metric. The answer to (C) is that the extension exists, but it takes values in the space of hyperplanes. What the image omits is no longer points but complex hyperplanes, and the geometric conclusion is the weak linear convexity of the image (Section~\ref{sec:AW}).

The rest of the chapter explores the limits of this picture. In the polydisk (Section~\ref{sec:polidisco}) the automorphisms are not Möbius, the invariance is lost and one must argue differently; in exchange, the hypothesis turns out to be surprisingly rigid. Section~\ref{sec:convexas} computes the norm of the Schwarzian of convex mappings. The chapter ends with historical notes, a list of open problems and some exercises.

The Schwarzian also extends in another direction, which we shall not pursue here. In the plane, Hernández and Martín \cite{HM15} defined the pre-Schwarzian and the Schwarzian of a harmonic mapping $f=h+\overline g$, in which the dilatation $\omega=g'/h'$ intervenes, and proved that a bound on the norm of $S_f$ implies univalence and quasiconformal extension \cite{HM15b}. The natural extension of $S_f$ to several variables, in the context of complex harmonic mappings, is the Schwarzian of the pluriharmonic mappings $f=h+\overline g$ of $\C^n$, defined in \cite{EFHV}; when $g=0$ it reduces to the operator $S_F$ of this chapter. Here we restrict ourselves to the holomorphic case.

A word of warning before we begin. Almost none of the constants that will appear is sharp, and in several variables the extremal mappings are rarely known. We have preferred to give explicit constants, even if they can be improved, and to say clearly which are sharp and which are not.

\section{Preliminaries in $\C^n$}\label{sec:prelim}

We collect here the notation and the general facts that will be used. The reader familiar with the book of Graham and Kohr \cite{GK} or with that of Rudin \cite{Rudin} may pass directly to the next section.

\medskip

\noindent\textbf{Notation.} We write $z=(z_1,\dots,z_n)\in\C^n$, with Hermitian product $z\cdot w=\sum_iz_i\overline{w_i}$ and Euclidean norm $|z|=(z\cdot z)^{1/2}$. We shall also use the \emph{bilinear} product
\[
\bil{a}{v}=\sum_{i=1}^na_iv_i,
\]
without conjugation, which is the one that appears when a holomorphic gradient acts on a vector, and the norm $|z|_\infty=\max_i|z_i|$. The unit ball and the polydisk are
\[
\B=\{z\in\C^n:\ |z|<1\},\qquad\Pn=\{z\in\C^n:\ |z|_\infty<1\}.
\]
We denote by $L^k(\C^n)$ the space of $k$-linear mappings from $(\C^n)^k$ into $\C^n$, with the norm $\|T\|=\sup\{|T(u_1,\dots,u_k)|:\ |u_j|=1\}$. For a matrix $A$, $\overline A$ is the matrix of conjugate entries, $A^t$ the transpose and $A^*=\overline A^{\,t}$ the adjoint.

Let $\Omega\subset\C^n$ be a domain and $F=(f_1,\dots,f_n):\Omega\to\C^n$ holomorphic. Its differential $DF(z)$ is the matrix $(\partial f_i/\partial z_j)$, its second derivative $D^2F(z)\in L^2(\C^n)$ is the symmetric bilinear form with components $\sum_{i,j}\partial^2f_l/\partial z_i\partial z_j\,u_iv_j$, and $J_F=\det DF$ is the complex Jacobian. The mapping is \emph{locally biholomorphic} if $J_F$ does not vanish, and \emph{biholomorphic} or \emph{univalent} if in addition it is injective. For holomorphic matrix-valued functions we have the rules
\begin{equation}\label{eq:producto}
D(\Phi\Psi)(v)=D\Phi(v)\,\Psi+\Phi\,D\Psi(v),\qquad D(\Phi^{-1})(v)=-\Phi^{-1}\,D\Phi(v)\,\Phi^{-1},
\end{equation}
and Jacobi's formula $\partial_i\log\det\Phi=\tr(\Phi^{-1}\partial_i\Phi)$, which applied to $\Phi=DF$ gives
\begin{equation}\label{eq:jacobi}
\pd{}{z_i}\log J_F=\tr\bigl(DF^{-1}D^2F(e_i,\cdot)\bigr).
\end{equation}

There is one difference from the plane that conditions everything that follows. For $n\ge2$ there is no Riemann mapping theorem, the ball and the polydisk are not biholomorphic to each other, and the class of normalized biholomorphic mappings of $\B$ is not normal. It suffices to look at $F(z_1,z_2)=(z_1+\varphi(z_2),z_2)$ with $\varphi$ an arbitrary holomorphic function, $\varphi(0)=\varphi'(0)=0$. For this reason geometric function theory in several variables works with subclasses, and results such as the bound $\|Sf\|\le6$ of Kraus and Nehari can have no analogue.

\medskip

\noindent\textbf{The ball, the polydisk and the Bergman metric.} The automorphisms of $\B$ act transitively and are all linear fractional transformations
\begin{equation}\label{eq:autbola}
\sigma(z)=\frac{Az+B}{Cz+D},
\end{equation}
with $A$ of size $n\times n$, $B$ of size $n\times1$, $C$ of size $1\times n$ and $D$ a scalar; the isotropy group of the origin is the unitary group \cite{Rudin}. In the polydisk, up to permutations and rotations of the coordinates, the automorphisms are
\begin{equation}\label{eq:autpoli}
\psi_a(z)=\bigl(\psi_{a_1}(z_1),\dots,\psi_{a_n}(z_n)\bigr),\qquad\psi_{a_j}(\zeta)=\frac{\zeta-a_j}{1-\overline{a_j}\zeta},\quad a\in\Pn .
\end{equation}
The group is transitive on $\Pn$ but not on the directions at a point, and for $n\ge2$ the mapping $\psi_a$ does not have the form \eqref{eq:autbola}, with a common denominator, unless $a=0$.

The Bergman metric of $\B$ is
\begin{equation}\label{eq:bergmanbola}
g_{ij}(z)=\frac{n+1}{(1-|z|^2)^2}\Bigl[(1-|z|^2)\delta_{ij}+\overline{z_i}z_j\Bigr],\qquad\|v\|_z^2=\sum_{i,j}g_{ij}(z)v_i\overline{v_j}.
\end{equation}
At the origin it is $(n+1)$ times the Euclidean metric. At the point $(z_1,0,\dots,0)$,
\[
\|v\|_z^2=(n+1)\Bigl[\frac{|v_1|^2}{(1-|z_1|^2)^2}+\frac{|v_2|^2+\cdots+|v_n|^2}{1-|z_1|^2}\Bigr],
\]
so that the radial direction and the tangential ones are measured with different weights. This anisotropy has no counterpart in the disk and will reappear several times. The geodesics through the origin are the diameters, and the automorphisms are isometries. In $\Pn$ the Bergman metric is diagonal,
\begin{equation}\label{eq:bergmanpoli}
\|v\|_z^2=2\sum_{i=1}^n\frac{|v_i|^2}{(1-|z_i|^2)^2},
\end{equation}
and the automorphisms \eqref{eq:autpoli} are also isometries. For $n=1$ both formulas reduce to $\|v\|_z^2=2|v|^2/(1-|z|^2)^2$, twice the square of the hyperbolic norm.

\medskip

\noindent\textbf{Möbius transformations.} We shall give this name to the projective transformations
\begin{equation}\label{eq:mobius}
M(z)=\Bigl(\frac{l_1(z)}{l_0(z)},\dots,\frac{l_n(z)}{l_0(z)}\Bigr),\qquad l_i(z)=a_{i0}+a_{i1}z_1+\cdots+a_{in}z_n,
\end{equation}
with $\det(a_{ij})_{i,j=0}^n\ne0$. They carry complex hyperplanes to complex hyperplanes. The automorphisms of the ball are of this type; those of the polydisk, in general, are not. The Jacobian is
\begin{equation}\label{eq:jacmobius}
J_M(z)=\frac{\det(a_{ij})}{l_0(z)^{n+1}},
\end{equation}
an identity from which comes the exponent $1/(n+1)$ that will appear everywhere. We shall frequently use the family
\begin{equation}\label{eq:Ta}
T_a(w)=\frac{w}{1+\bil{a}{w}},\qquad a\in\C^n,
\end{equation}
which fixes the origin, has identity differential there, and satisfies $J_{T_a}(w)=(1+\bil{a}{w})^{-(n+1)}$.

\medskip

\noindent\textbf{Linearly invariant families.} Let $\mathcal{LS}(\B)$ be the class of locally biholomorphic mappings $F:\B\to\C^n$ with $F(0)=0$ and $DF(0)=\Id$. For $\sigma\in\Aut(\B)$, the \emph{Koebe transform}
\begin{equation}\label{eq:koebe}
\Lambda_\sigma(F)(z)=\bigl[DF(\sigma(0))\,D\sigma(0)\bigr]^{-1}\bigl[F(\sigma(z))-F(\sigma(0))\bigr]
\end{equation}
belongs again to $\mathcal{LS}(\B)$, and a family $\mathcal F\subset\mathcal{LS}(\B)$ is \emph{linearly invariant} if it is closed under all of them. In several variables the coefficient $a_2$ is a bilinear form, $A_2=\tfrac12D^2F(0)$, and there are two natural ways of measuring it. The \emph{order} of Pfaltzgraff \cite{Pf97} uses the trace,
\begin{equation}\label{eq:orden}
\ord\mathcal F=\sup_{F\in\mathcal F}\ \sup_{|v|=1}\Bigl|\tr\Bigl(\tfrac12D^2F(0)(v,\cdot)\Bigr)\Bigr|=\tfrac12\sup_{F\in\mathcal F}|\nabla J_F(0)|,
\end{equation}
where the second equality is \eqref{eq:jacobi}. The \emph{norm order} of Pfaltzgraff and Suffridge \cite{PS} uses the norm,
\begin{equation}\label{eq:ordennorma}
\|\ord\|\,\mathcal F=\sup_{F\in\mathcal F}\ \sup_{|v|\le1}|A_2(F)(v,v)|.
\end{equation}
We have $\ord\mathcal F\le n\,\|\ord\|\,\mathcal F$ \cite[p.~433]{GK}. The order controls the Jacobian. If $\ord\mathcal F=b<\infty$, then for every $F\in\mathcal F$
\[
\frac{(1-|z|)^{b-\frac{n+1}{2}}}{(1+|z|)^{b+\frac{n+1}{2}}}\le|J_F(z)|\le\frac{(1+|z|)^{b-\frac{n+1}{2}}}{(1-|z|)^{b+\frac{n+1}{2}}}
\]
\cite[Theorem~10.2.3]{GK}, a result due to Liu and to Pfaltzgraff \cite{Pf97}. The norm order is stronger. If it is finite, the family is normal and there are bounds for $|F|$ and $\|DF\|$ \cite{PS}.

\medskip

\noindent\textbf{Loewner chains.}\label{par:loewner} A \emph{Loewner chain} in $\B$ is a family $f(z,t)$, $t\ge0$, such that each $f(\cdot,t)$ is biholomorphic in $\B$, $f(0,t)=0$, $Df(0,t)=e^t\Id$, and $f(\B,s)\subset f(\B,t)$ for $s\le t$. The theory was carried over to several variables by Pfaltzgraff \cite{Pf74}. The result we shall use says, in essence, the following \cite[Theorem~8.1.6]{GK}. If $f(z,t)=e^tz+\cdots$ is holomorphic in $z$ and locally absolutely continuous in $t$, the family $\{e^{-t}f(\cdot,t)\}$ is locally bounded, and
\begin{equation}\label{eq:loewnerpde}
\pd{f}{t}(z,t)=Df(z,t)\,h(z,t),
\end{equation}
with $h(\cdot,t)$ holomorphic, $h(0,t)=0$, $Dh(0,t)=\Id$ and $\RE\bigl(h(z,t)\cdot z\bigr)>0$ for $z\ne0$, then $f(z,t)$ is a Loewner chain and in particular $f(\cdot,0)$ is univalent. With this technique Pfaltzgraff obtained the version of Becker's criterion in $\C^n$.

\begin{thm}[Pfaltzgraff \cite{Pf74}; see \cite{GK}, Theorem~8.4.1]\label{thm:pfaltzgraff}
If $F\in\mathcal{LS}(\B)$ satisfies
\[
(1-|z|^2)\,\bigl\|DF(z)^{-1}D^2F(z)(z,\cdot)\bigr\|\le1,\qquad z\in\B,
\]
then $F$ is univalent in $\B$.
\end{thm}

The operator $P_F=DF^{-1}D^2F$ that appears here is the \emph{pre-Schwarzian} of $F$. Theorem~\ref{thm:pfaltzgraff} shows that the ``Becker'' side of the theory passes to several variables without major difficulty. The ``Nehari'' side is the one that requires new ideas, and to it we devote the sections that follow.

\section{The Schwarzian derivative in several variables}\label{sec:schwarziana}

Before giving definitions we repeat the parameter count of Section~\ref{sec:modelo}. The Möbius transformations of $\C^n$ form the projective group $PGL(n+1,\C)$, of dimension
\[
(n+1)^2-1=n^2+2n .
\]
A locally biholomorphic mapping has at a point $n$ parameters of order zero, $n^2$ of first order and $n\cdot\tfrac{n(n+1)}{2}$ of second order. An operator of order two whose kernel were determined by the data of order zero and one would have a kernel of dimension $n^2+n$; one of order three, with the data free up to order two, would have a kernel of dimension $n^2+n+\tfrac12n^2(n+1)$. For $n=1$ the second number is $3$, the dimension of the Möbius group, and everything fits. For $n\ge2$, however,
\begin{equation}\label{eq:recuento}
n^2+n\ <\ n^2+2n\ <\ n^2+n+\tfrac12\,n^2(n+1).
\end{equation}
The Möbius transformations have more degrees of freedom than an operator of order two allows, and fewer than an operator of order three allows. Hence there cannot be a Schwarzian that is an operator of definite order with all the data of lower order free. The correct definition will have to be a system of order two that fixes \emph{almost all} the second derivatives and leaves exactly $n$ of them free. We shall see that this is precisely what happens, and which are the $n$ parameters left over.

To answer question (A) we start from the equation and not from the formula. In one variable, the equation $u''+pu=0$ has a solution space of dimension two, and the quotient of two solutions determines a mapping to the projective line. To obtain a mapping with values in $\C^n$, or better, in the projective space of dimension $n$, one needs a solution space of dimension $n+1$. The natural candidates are the linear systems of second order
\begin{equation}\label{eq:sistema}
\pdd{u}{z_i}{z_j}=\sum_{k=1}^nP^k_{ij}(z)\pd{u}{z_k}+P^0_{ij}(z)\,u,\qquad i,j=1,\dots,n,
\end{equation}
with coefficients holomorphic in a domain $\Omega\subset\C^n$, symmetric in $i,j$. A solution of \eqref{eq:sistema} is determined by the $n+1$ numbers $u(z_0)$, $\nabla u(z_0)$, since the system gives the second derivatives in terms of them, and all the others are obtained by differentiating it. The solution space therefore has dimension at most $n+1$, and the system is said to be \emph{completely integrable} when the dimension is exactly $n+1$. It is said to be in \emph{canonical form} \cite{Yos76} if
\begin{equation}\label{eq:canonica}
\sum_{j=1}^nP^j_{ij}=0,\qquad i=1,\dots,n.
\end{equation}

Now let $F=(f_1,\dots,f_n):\Omega\to\C^n$ be locally biholomorphic, and write $w=F(z)$. Following Oda \cite{Oda}, we define for $i,j,k=1,\dots,n$
\begin{equation}\label{eq:oda}
S^k_{ij}F=\sum_{l=1}^n\pdd{f_l}{z_i}{z_j}\,\pd{z_k}{w_l}-\frac{1}{n+1}\Bigl(\delta^k_i\pd{}{z_j}+\delta^k_j\pd{}{z_i}\Bigr)\log J_F,
\end{equation}
and
\begin{equation}\label{eq:S0}
S^0_{ij}F=J_F^{\frac{1}{n+1}}\Bigl[\pdd{}{z_i}{z_j}J_F^{-\frac{1}{n+1}}-\sum_{k=1}^n\pd{}{z_k}J_F^{-\frac{1}{n+1}}\,S^k_{ij}F\Bigr].
\end{equation}
The first term of \eqref{eq:oda} is the $k$-th component of the vector $DF^{-1}D^2F(e_i,e_j)$, and the second is chosen so that \eqref{eq:canonica} holds. Indeed, by \eqref{eq:jacobi}, $\sum_jS^j_{ij}F=\partial_i\log J_F-\frac{1+n}{n+1}\,\partial_i\log J_F=0$. The definition \eqref{eq:S0}, for its part, says simply that $J_F^{-1/(n+1)}$ is a solution of the system with coefficients $S^k_{ij}F$ and $S^0_{ij}F$. The following theorem justifies both definitions.

\begin{thm}[Oda]\label{thm:oda}
Let $\Omega\subset\C^n$ be simply connected.
\begin{enumerate}
\item[(a)] If $F:\Omega\to\C^n$ is locally biholomorphic, the system \eqref{eq:sistema} with $P^k_{ij}=S^k_{ij}F$, $0\le k\le n$, is completely integrable and is in canonical form. The functions
\begin{equation}\label{eq:base}
u_0=J_F^{-\frac{1}{n+1}},\qquad u_i=f_i\,J_F^{-\frac{1}{n+1}}\quad(i=1,\dots,n)
\end{equation}
form a basis of solutions, and $F=(u_1/u_0,\dots,u_n/u_0)$.
\item[(b)] Conversely, if \eqref{eq:sistema} is completely integrable and is in canonical form, and $u_0,\dots,u_n$ is a basis of solutions, then $F=(u_1/u_0,\dots,u_n/u_0)$ is locally biholomorphic where $u_0\ne0$, and $P^k_{ij}=S^k_{ij}F$ for $0\le k\le n$.
\end{enumerate}
\end{thm}

\begin{proof}
(a) Let $u_0=J_F^{-1/(n+1)}$, which is a solution by \eqref{eq:S0}. For $u=f_lu_0$ we have
\[
\partial_{ij}u=u_0\,\partial_{ij}f_l+\partial_if_l\,\partial_ju_0+\partial_jf_l\,\partial_iu_0+f_l\,\partial_{ij}u_0 ,
\]
and since $u_0$ is a solution, $u$ will be one if and only if
\begin{equation}\label{eq:clave}
u_0\,\partial_{ij}f_l+\partial_if_l\,\partial_ju_0+\partial_jf_l\,\partial_iu_0=u_0\sum_kS^k_{ij}F\,\partial_kf_l .
\end{equation}
Now $\sum_k\bigl[DF^{-1}D^2F(e_i,e_j)\bigr]_k\,\partial_kf_l=\partial_{ij}f_l$, so that by \eqref{eq:oda}
\[
\sum_kS^k_{ij}F\,\partial_kf_l=\partial_{ij}f_l-\frac{1}{n+1}\bigl(\partial_if_l\,\partial_j\log J_F+\partial_jf_l\,\partial_i\log J_F\bigr),
\]
and \eqref{eq:clave} follows from $\partial_ju_0=-\frac{1}{n+1}u_0\,\partial_j\log J_F$. The $n+1$ functions \eqref{eq:base} are linearly independent, since a relation $c_0+\sum c_if_i\equiv0$ gives, upon differentiation, $\sum c_i\nabla f_i=0$, and $DF$ is invertible. The solution space then has dimension $n+1$.

(b) Consider the determinant $W$ of the matrix of order $n+1$ whose rows are the vectors $(u_j,\nabla u_j)$, $j=0,\dots,n$. It does not vanish, because a solution is determined by its value and its gradient at a point. When $W$ is differentiated with respect to $z_i$ column by column, the column of the $u_j$ contributes nothing, and the column of the $\partial_mu_j$, replaced by means of \eqref{eq:sistema}, contributes $P^m_{im}W$. Thus $\partial_iW=\bigl(\sum_mP^m_{im}\bigr)W=0$ by \eqref{eq:canonica}, and $W$ is constant. On the other hand, writing $u_j=u_0f_j$ and subtracting from each row a multiple of the first, we obtain $W=u_0^{\,n+1}J_F$. Consequently $J_F\ne0$ and $u_0=c\,J_F^{-1/(n+1)}$ with $c$ constant. By part (a), the system associated with $F$ has as solutions $J_F^{-1/(n+1)}(1,f_1,\dots,f_n)$, that is, the multiples of $u_0,\dots,u_n$. The two systems have the same solution space, and a completely integrable system is determined by it. At each point the vectors $(u(z),\nabla u(z))$, with $u$ a solution, fill $\C^{n+1}$, and \eqref{eq:sistema} expresses the Hessian as a linear function of that vector. The coefficients coincide.
\end{proof}

The role played in one variable by $(f')^{-1/2}$ is now played by $J_F^{-1/(n+1)}$, and the exponent is explained by \eqref{eq:jacmobius}. If $M$ is a Möbius transformation with denominator $l_0$, then $J_{M\circ F}^{-1/(n+1)}$ is a constant multiple of $l_0(F)\,J_F^{-1/(n+1)}$, which is a linear combination of the solutions \eqref{eq:base}. From this the two properties we were seeking follow at once.

\begin{cor}\label{cor:mobius}
If $M$ is a Möbius transformation, $S^k_{ij}(M\circ F)=S^k_{ij}F$ for $0\le k\le n$. Moreover, for $n\ge2$,
\[
S^k_{ij}F\equiv0\ \text{ for all }i,j,k\ge1\iff F\ \text{is a Möbius transformation.}
\]
\end{cor}

\begin{proof}
The systems associated with $F$ and with $M\circ F$ have the same solution space, hence the same coefficients. For $F(z)=z$ all the coefficients vanish, and therefore they also vanish for every Möbius transformation. Conversely, suppose $S^k_{ij}F=0$ for $k\ge1$, so that the system is $\partial_{ij}u=S^0_{ij}F\,u$. Equating $\partial_k(\partial_{ij}u)$ with $\partial_j(\partial_{ik}u)$ for an index $k\ne j$, which is possible because $n\ge2$, and comparing the coefficients of $\partial_ku$, which at a given point is arbitrary, we obtain $S^0_{ij}F=0$. The system then reduces to $\Hess u=0$, its solutions are the affine functions, and $F$ is a quotient of affine functions.
\end{proof}

\begin{cor}\label{cor:soluciones}
Every nontrivial solution $u$ of the system associated with $F$ is of the form $u=J_G^{-1/(n+1)}$ with $G=T\circ F$ and $T$ a Möbius transformation. The mapping $G$ is holomorphic off the zero set of $u$.
\end{cor}

The count \eqref{eq:recuento} is now explained. The $S^k_{ij}F$ with $k\ge1$ are $n\cdot\tfrac{n(n+1)}{2}$ functions subject to the $n$ relations \eqref{eq:canonica}, so that the system $S^k_{ij}F=0$ consists of $\tfrac12n(n-1)(n+2)$ independent equations of second order. It fixes all the second derivatives of $F$ at a point except for $n$ combinations, the components of $\nabla\log J_F$, which the second term of \eqref{eq:oda} deliberately leaves free. The initial data of a solution are then $F(z_0)$, $DF(z_0)$ and $\nabla\log J_F(z_0)$, in total
\[
n+n^2+n=n^2+2n
\]
parameters, which is the dimension of the Möbius group. The Schwarzian in several variables is an operator of second order that does not see the trace, and for that reason it lies between the two extremes of \eqref{eq:recuento}.

For a general composition one has the chain rule \cite{Oda}
\begin{equation}\label{eq:cadena}
S^k_{ij}(G\circ F)(z)=S^k_{ij}F(z)+\sum_{l,m,r=1}^nS^r_{lm}G(w)\,\pd{w_l}{z_i}\pd{w_m}{z_j}\pd{z_k}{w_r},\qquad w=F(z).
\end{equation}

\begin{obs}[The case $n=1$]\label{obs:n1}
For $n=1$ formula \eqref{eq:oda} gives $S^1_{11}f=f''/f'-f''/f'=0$ for every $f$, and the system reduces to $u''=S^0_{11}f\,u$ with $u_0=(f')^{-1/2}$. Therefore
\[
S^0_{11}f=-\tfrac12\,Sf .
\]
The classical Schwarzian is the coefficient of \emph{order zero}, an operator of third order. For $n\ge2$ the situation is reversed. The $S^k_{ij}F$ with $k\ge1$ are operators of second order, nontrivial, and it is they that characterize the Möbius transformations. The $S^0_{ij}F$ are no longer free data. The equality of the mixed derivatives $\partial_k(\partial_{ij}u)=\partial_j(\partial_{ik}u)$ determines them in terms of the former \cite{H07}, through the formulas
\begin{align}
S^0_{ii}F&=\frac{1}{n-1}\Bigl[\sum_{j,k=1}^nS^k_{ij}F\,S^j_{ki}F-\sum_{k=1}^n\pd{}{z_k}S^k_{ii}F\Bigr],\label{eq:S0ii}\\
S^0_{ij}F&=\pd{}{z_j}S^i_{ii}F-\pd{}{z_i}S^i_{ij}F+\sum_{k=1}^n\bigl(S^k_{ii}F\,S^i_{kj}F-S^k_{ij}F\,S^i_{ki}F\bigr),\qquad j\ne i.\label{eq:S0ij}
\end{align}
The factor $1/(n-1)$ shows why the case $n=1$ is excluded. For $n=1$ there are no equations of second order, the only second derivative remains free, and an operator of third order is needed; this is the case of equality in \eqref{eq:recuento}. In several variables, by contrast, the information lies in an object of second order, closer to a pre-Schwarzian than to a Schwarzian, and this will have to be kept in mind when interpreting the norms.
\end{obs}

\medskip

\noindent\textbf{The operator $S_F$.} The coefficients $S^k_{ij}F$ depend on the coordinates, and in order to measure them it is convenient to gather them into an intrinsic object. For each $k\ge0$ let $S^kF=(S^k_{ij}F)_{i,j}$, a symmetric $n\times n$ matrix.

\begin{defn}[\cite{H06}]\label{def:operador}
The \emph{Schwarzian operator} of $F$ at $z$ is the symmetric bilinear mapping
\[
S_F(z)(v,w)=\bigl(v^tS^1F(z)w,\dots,v^tS^nF(z)w\bigr),\qquad v,w\in\C^n .
\]
We also set $S^0_F(z)(v,w)=v^tS^0F(z)w$, with scalar values.
\end{defn}

Reading \eqref{eq:oda} component by component we obtain the expression
\begin{equation}\label{eq:SP}
S_F(z)(v,w)=P_F(z)(v,w)-\frac{1}{n+1}\Bigl[\bil{\nabla\log J_F(z)}{v}\,w+\bil{\nabla\log J_F(z)}{w}\,v\Bigr],
\end{equation}
where $P_F=DF^{-1}D^2F$ is the pre-Schwarzian. This formula is the most convenient for computation, and it has a clear reading. By \eqref{eq:canonica}, $S_F$ is the trace-free part of $P_F$. What has been removed from the pre-Schwarzian is exactly what changes when $F$ is composed with a Möbius transformation. The system of Oda can now be written compactly,
\begin{equation}\label{eq:sistemacompacto}
\Hess u(z)(v,v)=\bil{\nabla u(z)}{S_F(z)(v,v)}+S^0_F(z)(v,v)\,u(z),
\end{equation}
and the chain rule \eqref{eq:cadena} takes the form
\begin{equation}\label{eq:cadenaop}
S_{G\circ F}(z)(v,w)=S_F(z)(v,w)+DF(z)^{-1}S_G(F(z))\bigl(DF(z)v,DF(z)w\bigr).
\end{equation}
In particular, if $\sigma$ is a Möbius transformation,
\begin{equation}\label{eq:precomp}
S_{F\circ\sigma}(z)(v,w)=D\sigma(z)^{-1}S_F(\sigma(z))\bigl(D\sigma(z)v,D\sigma(z)w\bigr).
\end{equation}
This identity says that $S_F$ transforms as a tensor of type $(1,2)$ under Möbius changes of variable, that is, it takes two tangent vectors at $z$ and returns one.

Two examples, which we shall use repeatedly, show what $S_F$ sees and what it does not.

\begin{ex}\label{ex:producto}
Let $F(z)=(f(z_1),z_2,\dots,z_n)$ with $f'\ne0$. Then $J_F=f'(z_1)$ and \eqref{eq:SP} gives
\[
S^1_{11}F=\frac{n-1}{n+1}\,\frac{f''}{f'}(z_1),\qquad S^i_{1i}F=-\frac{1}{n+1}\,\frac{f''}{f'}(z_1)\quad(i\ge2),
\]
and the remaining coefficients vanish. Here the operator $S_F$ records the pre-Schwarzian of $f$ and not its Schwarzian.
\end{ex}

\begin{ex}\label{ex:RS}
The extension operator of Roper and Suffridge \cite{RS} associates with a locally univalent function $f$ of the disk, with $f(0)=0$ and $f'(0)=1$, the mapping
\begin{equation}\label{eq:RS}
\Phi_f(z)=\bigl(f(z_1),\sqrt{f'(z_1)}\,z_2,\dots,\sqrt{f'(z_1)}\,z_n\bigr).
\end{equation}
Here $J_{\Phi_f}=(f')^{(n+1)/2}$, and a computation with \eqref{eq:SP} shows that the only nonzero coefficients are
\[
S^k_{11}\Phi_f(z)=\tfrac12\,z_k\,Sf(z_1),\qquad k=2,\dots,n.
\]
Now the classical Schwarzian does appear, and $S_{\Phi_f}\equiv0$ exactly when $f$ is a Möbius transformation.
\end{ex}

\medskip

\noindent\textbf{The Schwarzian as curvature.}\label{sec:curvatura} Corollary~\ref{cor:mobius} has a geometric reading. Möbius transformations carry hyperplanes to hyperplanes, so that if $S_F\equiv0$ the preimage under $F$ of every hyperplane is a piece of a hyperplane. One expects that, in general, $S_F$ measures how much those preimages curve. This is so, and the relation is an identity.

Let $H=\{w\in\C^n:\ a_0+a_1w_1+\dots+a_nw_n=0\}$ be a hyperplane that meets $F(\Omega)$. Since $u_0$ does not vanish, the formulas \eqref{eq:base} give
\begin{equation}\label{eq:Sigma}
\Sigma=F^{-1}(H)=\{z\in\Omega:\ u(z)=0\},\qquad u=a_0u_0+a_1u_1+\dots+a_nu_n .
\end{equation}
The preimages of the hyperplanes are therefore the zero sets of the solutions of the system associated with $F$. The gradient of $u$ does not vanish at any point of $\Sigma$, because a solution with $u(z_0)=0$ and $\nabla u(z_0)=0$ is identically zero. Hence $\Sigma$ is a smooth complex hypersurface. Its tangent space at $z$ is $\{v:\bil{\nabla u(z)}{v}=0\}$, and its normal space is the complex line spanned by the unit vector
\[
N(z)=\frac{\overline{\nabla u(z)}}{|\nabla u(z)|} .
\]
We regard $\Sigma$ as a real submanifold of $\C^n=\R^{2n}$ with the Euclidean metric. Given $z\in\Sigma$ and a unit tangent vector $v$, let $\gamma$ be the geodesic of $\Sigma$ with $\gamma(0)=z$ and $\gamma'(0)=v$, parametrized by arc length. Its curvature as a curve in $\R^{2n}$,
\[
\kappa_\Sigma(z,v)=|\gamma''(0)|,
\]
is the \emph{extrinsic curvature} of $\Sigma$ in the direction $v$, and it measures how fast $\Sigma$ departs from its tangent plane as one moves in that direction.

\begin{thm}\label{thm:curvatura}
Let $F:\Omega\to\C^n$ be locally biholomorphic, $n\ge2$, and let $\Sigma=F^{-1}(H)$ be as in \eqref{eq:Sigma}. For every $z\in\Sigma$ and every unit vector $v$ tangent to $\Sigma$ at $z$,
\begin{equation}\label{eq:curvatura}
\kappa_\Sigma(z,v)=\frac{\bigl|\bil{\nabla u(z)}{S_F(z)(v,v)}\bigr|}{|\nabla u(z)|} .
\end{equation}
That is, $\kappa_\Sigma(z,v)$ is the modulus of the component of $S_F(z)(v,v)$ in the normal direction $N(z)$. In particular,
\begin{equation}\label{eq:curvaturacota}
\kappa_\Sigma(z,v)\le|S_F(z)(v,v)| .
\end{equation}
\end{thm}

\begin{proof}
Since $u$ is holomorphic and vanishes along $\gamma$, differentiating $u(\gamma(t))=0$ twice we obtain
\[
\bil{\nabla u(\gamma)}{\gamma'}=0,\qquad\Hess u(\gamma)(\gamma',\gamma')+\bil{\nabla u(\gamma)}{\gamma''}=0 .
\]
The acceleration of a geodesic is normal to the surface, so that $\gamma''(0)=c\,N(z)$ with $c\in\C$, and then $\bil{\nabla u(z)}{\gamma''(0)}=c\,|\nabla u(z)|$. Therefore
\[
\kappa_\Sigma(z,v)=|c|=\frac{|\Hess u(z)(v,v)|}{|\nabla u(z)|} .
\]
On the other hand $u(z)=0$, and the system \eqref{eq:sistemacompacto} reduces at that point to $\Hess u(z)(v,v)=\bil{\nabla u(z)}{S_F(z)(v,v)}$. This proves \eqref{eq:curvatura}, and \eqref{eq:curvaturacota} is the Cauchy--Schwarz inequality.
\end{proof}

\begin{figure}[ht]
\centering
\begin{tikzpicture}[>=stealth,scale=1]
\draw[gray!70] (0,0) circle (2.1);
\node[gray!60!black] at (-1.55,1.75) {$\Omega$};
\draw[thick] plot[smooth,tension=.8] coordinates {(-1.75,-1.05) (-0.9,-0.15) (0,0.25) (0.95,0.1) (1.85,-0.75)};
\node at (1.5,-0.95) {$\Sigma$};
\fill (0,0.25) circle (1.4pt);
\node[below] at (0,0.2) {$z$};
\draw[->] (0,0.25)--(1.15,0.37) node[right] {$v$};
\draw[->,dashed] (0,0.25)--(-0.13,1.45) node[above] {$N(z)$};
\draw[->,thick] (0,0.25)--(0.85,1.25);
\node[right] at (0.85,1.3) {$S_F(z)(v,v)$};
\draw[densely dotted] (0.85,1.25)--(-0.105,1.215);
\draw[very thick] (0,0.25)--(-0.105,1.215);
\node[left] at (-0.1,0.8) {$\kappa_\Sigma$};
\draw[->] (2.6,0.3) to[bend left=18] node[above] {$F$} (4.3,0.3);
\draw[gray!70] plot[smooth cycle,tension=.75] coordinates {(4.9,-0.3) (5.5,1.5) (7.2,1.9) (8.6,1.0) (8.5,-0.9) (7.0,-1.8) (5.5,-1.4)};
\node[gray!60!black] at (6.25,1.0) {$F(\Omega)$};
\draw[thick] (4.75,-1.25)--(8.95,0.85);
\node at (9.2,0.95) {$H$};
\fill (6.85,-0.2) circle (1.4pt);
\node[below right] at (6.85,-0.2) {$F(z)$};
\end{tikzpicture}
\caption{The preimage $\Sigma=F^{-1}(H)$ of a hyperplane. Its extrinsic curvature at $z$, in the direction $v$, is the normal component of $S_F(z)(v,v)$.}\label{fig:curvatura}
\end{figure}
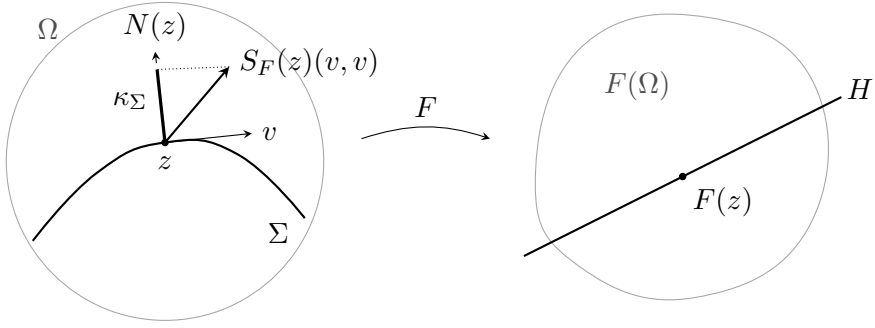

Three remarks on the proof. The first is that the coefficient of order zero disappears because $u$ vanishes on $\Sigma$, so that the curvature of the preimages depends only on the $S^k_{ij}F$ with $k\ge1$, which for $n\ge2$ are the ones that characterize the Möbius transformations. The second is that in \eqref{eq:curvatura} one can replace $S_F$ by the pre-Schwarzian $P_F$. By \eqref{eq:SP}, the two vectors $S_F(z)(v,v)$ and $P_F(z)(v,v)$ differ by a multiple of $v$, which is tangent to $\Sigma$ and has no normal component. The curvature sees only the vector $S_F(z)(v,v)$ modulo $v$. The third is that $\kappa_\Sigma(z,e^{i\theta}v)=\kappa_\Sigma(z,v)$, that is, the curvature depends on the complex direction and not on the real direction within it.

With $z\in\Omega$ and a unit vector $v$ fixed, every vector $a\ne0$ with $\bil{a}{v}=0$ is the gradient at $z$ of a solution that vanishes at $z$, unique up to a constant factor. Hence there is a preimage of a hyperplane through $z$ tangent to $v$ for each admissible normal direction, and letting $a$ vary in \eqref{eq:curvatura} we obtain
\begin{equation}\label{eq:curvmax}
\max_{\Sigma}\,\kappa_\Sigma(z,v)=\bigl|S_F(z)(v,v)-\lambda v\bigr|,\qquad\lambda=\bil{S_F(z)(v,v)}{\overline v},
\end{equation}
where the maximum is taken over the preimages of hyperplanes that pass through $z$ and are tangent to $v$. The right-hand side is the distance from $S_F(z)(v,v)$ to the line $\C v$. To know how much information about $S_F$ these curvatures contain, the following algebraic lemma then suffices, in which the condition \eqref{eq:canonica} appears once more.

\begin{lem}\label{lem:modv}
Let $n\ge2$ and let $Q:\C^n\times\C^n\to\C^n$ be bilinear and symmetric, with coefficients $Q^k_{ij}$ such that $\sum_jQ^j_{ij}=0$ for all $i$. If $Q(v,v)$ is a multiple of $v$ for every $v\in\C^n$, then $Q=0$.
\end{lem}

\begin{proof}
Write $Q^k(v)$ for the $k$-th component of $Q(v,v)$, which is a homogeneous polynomial of degree two. The hypothesis is equivalent to $Q^k(v)\,v_i=Q^i(v)\,v_k$ for all $i,k$. If $i\ne k$, the polynomial $v_k$ divides $Q^k(v)\,v_i$ and therefore $Q^k(v)$, so that $Q^k(v)=\ell_k(v)\,v_k$ with $\ell_k$ linear; substituting, $\ell_k=\ell_i$. Thus $Q(v,v)=\ell(v)\,v$ with $\ell(v)=\sum_i\ell_iv_i$, and by polarization $Q^k_{ij}=\tfrac12(\ell_i\delta^k_j+\ell_j\delta^k_i)$. Then $\sum_jQ^j_{ij}=\tfrac12(n+1)\ell_i$, and the hypothesis on the trace gives $\ell=0$.
\end{proof}

\begin{cor}\label{cor:curvatura}
Let $F:\Omega\to\C^n$ be locally biholomorphic, $n\ge2$, and $z\in\Omega$. Then $S_F(z)=0$ if and only if every preimage of a hyperplane that passes through $z$ has zero extrinsic curvature at $z$, in all tangent directions. Consequently, $F$ is a Möbius transformation if and only if the preimages of the hyperplanes are totally geodesic.
\end{cor}

\begin{proof}
If the curvatures vanish, \eqref{eq:curvmax} says that $S_F(z)(v,v)$ is a multiple of $v$ for every $v$, and Lemma~\ref{lem:modv} applies to $Q=S_F(z)$, which is trace-free by \eqref{eq:canonica}. The converse is \eqref{eq:curvaturacota}, and the last assertion follows from Corollary~\ref{cor:mobius}.
\end{proof}

The lemma says something more. The mapping $v\mapsto S_F(z)(v,v)$ modulo $v$, which is what the curvatures see, determines the operator $S_F(z)$. The trace-free normalization, which in the definition of Oda appeared as a condition on the system, is the one that selects in each class the unique representative.

In $\C^2$ all of this becomes explicit. The tangent space to $\Sigma$ is a complex line, the normal direction to a unit vector $v=(v_1,v_2)$ is $a=(v_2,-v_1)$, and there is no maximum to take in \eqref{eq:curvmax}. Using the relations $S^2_{12}F=-S^1_{11}F$ and $S^2_{22}F=-S^1_{12}F$, which are \eqref{eq:canonica} for $n=2$, formula \eqref{eq:curvatura} becomes
\begin{equation}\label{eq:cubica}
\kappa_\Sigma(z,v)=|c_z(v)|,\qquad c_z(v)=-S^2_{11}F\,v_1^3+3S^1_{11}F\,v_1^2v_2+3S^1_{12}F\,v_1v_2^2+S^1_{22}F\,v_2^3,
\end{equation}
with the coefficients evaluated at $z$. The curvature of the preimages of the lines is the modulus of a binary cubic form, and its four coefficients are, up to numerical factors, the four independent Schwarzians counted by the formula $\tfrac12n(n-1)(n+2)$ for $n=2$. Two cubic forms with the same modulus on all of $\C^2$ differ by a constant factor of modulus one, so the curvatures $\kappa_\Sigma(z,v)$ determine $S_F(z)$ up to that factor. In this case $\Sigma$ is moreover a minimal surface in $\R^4$, and its Gauss curvature is $K=-2\kappa_\Sigma^2$ (Exercise~\ref{ej:gauss}).

\medskip

\noindent\textbf{The Schwarzian and Halley's method.} There is another route to \eqref{eq:SP}, independent of differential equations, which explains in a different way the factor $J_F^{-1/(n+1)}$. In one variable, Newton's method for approximating a zero $a$ of $f$ iterates the function $g(z)=z-f(z)/f'(z)$, which satisfies $g(a)=a$, $g'(a)=0$ and $g''(a)=Pf(a)$. Halley's method is nothing other than Newton's method applied to $f\cdot(f')^{-1/2}$; its iteration function $h$ satisfies $h'(a)=h''(a)=0$ and $h'''(a)=-Sf(a)$. The pre-Schwarzian and the Schwarzian are thus the first nontrivial coefficients of two classical numerical methods. In $\C^n$ Newton's method iterates $z\mapsto z-DF(z)^{-1}F(z)$, and the analogue of Halley's method consists in applying it to $g=J_F^{-1/(n+1)}F$, that is, in iterating $G(z)=z-Dg(z)^{-1}g(z)$.

\begin{thm}[\cite{H22}]\label{thm:halley}
Let $F:\Omega\to\C^n$ be locally biholomorphic in a simply connected domain, with $F(a)=0$. Then
\[
G(a)=a,\qquad DG(a)=0,\qquad D^2G(a)=S_F(a).
\]
\end{thm}

\begin{proof}
From \eqref{eq:producto}, $DG(z)(v)=Dg(z)^{-1}D^2g(z)\bigl(v,Dg(z)^{-1}g(z)\bigr)$, which vanishes at $a$ because $g(a)=0$. Differentiating once more and evaluating at $a$, the terms that contain $g(a)$ disappear and we are left with $D^2G(a)=Dg(a)^{-1}D^2g(a)=P_g(a)$. On the other hand, at $a$ we have $Dg=J_F^{-1/(n+1)}DF$ and
\[
D^2g(v,w)=J_F^{-\frac{1}{n+1}}\Bigl[D^2F(v,w)-\frac{1}{n+1}\bigl(\bil{\nabla\log J_F}{v}\,DF\,w+\bil{\nabla\log J_F}{w}\,DF\,v\bigr)\Bigr],
\]
whence $P_g(a)=S_F(a)$ by \eqref{eq:SP}.
\end{proof}

For $n=1$ the theorem says $G''(a)=0$, which is the gain in order of Halley's method over Newton's, and the classical Schwarzian appears only in the next derivative. For $n\ge2$ that gain in order occurs only when $S_F(a)=0$. It is the same dichotomy as in Remark~\ref{obs:n1}, seen from another side.

\medskip

\noindent\textbf{Prescribing the pre-Schwarzian.}\label{sec:preschwarziana} Theorem~\ref{thm:oda} also answers a question about $P_F$ that in one variable is trivial. Which bilinear operators $A$ are pre-Schwarzians? For $n=1$ every holomorphic function $A$ is $f''/f'$ for $f=\int\exp\int A$. For $n\ge2$ the equation $DF^{-1}D^2F=A$ is an overdetermined system and there are integrability conditions.

\begin{thm}[\cite{H11}]\label{thm:unicidad}
Let $f,g$ be locally biholomorphic in a domain $\Omega\subset\C^n$. Then $P_f=P_g$ if and only if $f=Ag+b$ with $A$ a constant invertible matrix and $b\in\C^n$.
\end{thm}

\begin{proof}
For $n=1$ it suffices to integrate twice. Let $n\ge2$. If $P_f=P_g$, taking traces and using \eqref{eq:jacobi} we obtain $\nabla\log J_f=\nabla\log J_g$, and then $S_f=S_g$ by \eqref{eq:SP}. The coefficients of order zero also coincide, by Remark~\ref{obs:n1}, and the systems associated with $f$ and with $g$ are the same. By Theorem~\ref{thm:oda}, $f=T\circ g$ with $T$ a Möbius transformation. Since $J_f=(J_T\circ g)\,J_g$ and $J_f/J_g$ is constant, $J_T$ is constant, and by \eqref{eq:jacmobius} $T$ is affine. The converse is immediate.
\end{proof}

\begin{thm}[\cite{H11}]\label{thm:prescribir}
Let $A(z)=(a^k_{ij}(z))$ be a holomorphic bilinear operator in a simply connected domain $\Omega$. There exists a locally biholomorphic $f:\Omega\to\C^n$ with $P_f=A$ if and only if
\begin{enumerate}
\item[(i)] $a^k_{ij}=a^k_{ji}$ for all $i,j,k$;
\item[(ii)] there exists $\phi$ holomorphic in $\Omega$ with $\displaystyle\sum_{k=1}^na^k_{kj}=\pd{\phi}{z_j}$, $j=1,\dots,n$;
\item[(iii)] the system \eqref{eq:sistema} with coefficients $P^k_{ij}=a^k_{ij}-\frac{1}{n+1}\bigl(\delta^k_i\,\partial_j\phi+\delta^k_j\,\partial_i\phi\bigr)$ is completely integrable, and $\exp(-\phi/(n+1))$ is one of its solutions.
\end{enumerate}
\end{thm}

The three conditions can be read off directly. The first is the symmetry of $D^2f$. The second says that the trace of $A$ must be a gradient, that of $\phi=\log J_f$, and those are the $n$ parameters that the Schwarzian does not see. The third reconstructs $f$. The $P^k_{ij}$ are the candidates for $S^k_{ij}f$, Theorem~\ref{thm:oda} produces a mapping with that Schwarzian, unique up to a Möbius transformation, and the requirement that $\exp(-\phi/(n+1))$ be a solution selects, by Corollary~\ref{cor:soluciones}, the representative whose Jacobian is $e^\phi$.

One application of this theorem is the extension to $\C^n$ of the integral transform $f_\alpha(z)=\int_0^z(f')^\alpha$, which in one variable is characterized by $Pf_\alpha=\alpha\,Pf$. In \cite{H11} one defines $f_\alpha$ by the condition $P_{f_\alpha}=\alpha P_f$, when $\alpha P_f$ is a pre-Schwarzian, and it is proved that if the linearly invariant family generated by $f$ in $\B$ has norm order $b$, then $f_\alpha$ is univalent for $|\alpha|\le1/(2b+1)$. This is a consequence of Theorem~\ref{thm:pfaltzgraff}.

\section{The Bergman norm and the family $\Fa$}\label{sec:familia}

We turn to question (B). In one variable the norm \eqref{eq:normaclasica} is obtained by weighting $|Sf|$ with the hyperbolic metric, and it is invariant because the automorphisms of the disk are at once hyperbolic isometries and Möbius transformations. In $\B$ the place of the hyperbolic metric is taken by the Bergman metric. Since $S_F(z)(v,v)$ is again a tangent vector at $z$, it makes sense to measure the argument and the result with the same metric.

\begin{defn}\label{def:norma}
Let $\Omega$ be the ball or the polydisk, with its Bergman metric, and let $F:\Omega\to\C^n$ be locally biholomorphic. We define
\[
\|S_F(z)\|=\sup_{\|v\|_z=1}\|S_F(z)(v,v)\|_z,\qquad\|S_F\|=\sup_{z\in\Omega}\|S_F(z)\|.
\]
\end{defn}

\begin{thm}[\cite{H06}]\label{thm:invariancia}
If $F:\B\to\C^n$ is locally biholomorphic and $\sigma\in\Aut(\B)$, then $\|S_{F\circ\sigma}(z)\|=\|S_F(\sigma(z))\|$ for all $z\in\B$. In particular $\|S_{F\circ\sigma}\|=\|S_F\|$.
\end{thm}

\begin{proof}
Since $\sigma$ is a Möbius transformation, \eqref{eq:precomp} holds. Let $w=D\sigma(z)v$. Since $\sigma$ is an isometry, $\|w\|_{\sigma(z)}=\|v\|_z$ and
\[
\|S_{F\circ\sigma}(z)(v,v)\|_z=\|D\sigma(z)^{-1}S_F(\sigma(z))(w,w)\|_z=\|S_F(\sigma(z))(w,w)\|_{\sigma(z)}.
\]
It suffices to take the supremum over $\|v\|_z=1$.
\end{proof}

The proof uses the same two properties of the automorphisms as in the disk. That they are isometries is true in any bounded homogeneous domain with its Bergman metric; that they are Möbius transformations is special to the ball. In the polydisk this second property fails, and with it the invariance (Section~\ref{sec:polidisco}).

\begin{obs}\label{obs:curvaturabergman}
The Bergman norm controls the curvatures of Theorem~\ref{thm:curvatura}. By \eqref{eq:bergmanbola},
\[
\frac{(n+1)\,|v|^2}{1-|z|^2}\ \le\ \|v\|_z^2\ \le\ \frac{(n+1)\,|v|^2}{(1-|z|^2)^2},
\]
and applying the first inequality to $S_F(z)(v,v)$ and the second to $v$ we obtain, for $|v|=1$,
\[
|S_F(z)(v,v)|\le\frac{\sqrt{1-|z|^2}}{\sqrt{n+1}}\,\|S_F(z)\|\,\|v\|_z^2\le\frac{\sqrt{n+1}\,\|S_F(z)\|}{(1-|z|^2)^{3/2}} .
\]
By \eqref{eq:curvaturacota}, if $\|S_F\|\le\alpha$ in $\B$ then every preimage $\Sigma$ of a hyperplane satisfies
\begin{equation}\label{eq:curvaturabergman}
\kappa_\Sigma(z,v)\le\frac{\sqrt{n+1}\,\alpha}{(1-|z|^2)^{3/2}},\qquad z\in\Sigma .
\end{equation}
At the origin the two inequalities used are equalities.
\end{obs}

Fix $\alpha\ge0$ and consider the family
\begin{equation}\label{eq:Fa}
\Fa=\bigl\{F\in\mathcal{LS}(\B):\ \|S_F\|\le\alpha\bigr\}.
\end{equation}
The Koebe transforms \eqref{eq:koebe} are obtained by composing with an automorphism on the right and with an affine mapping on the left, so that Theorem~\ref{thm:invariancia} and Corollary~\ref{cor:mobius} show that $\Fa$ is linearly invariant. For $\alpha=0$ it consists of the Möbius transformations $T_a$ of \eqref{eq:Ta} that are holomorphic in the ball, which occurs if and only if $|a|\le1$. The program, from here on, is that of Pommerenke, which consists in showing that $\Fa$ is compact and estimating its order.

\medskip

\noindent\textbf{Pointwise bounds and normality.} The hypothesis $\|S_F\|\le\alpha$ is a single inequality on an object with many components, and the first step is to take it apart.

\begin{lem}[\cite{H06}]\label{lem:cotas}
Let $F:\B\to\C^n$ with $\|S_F\|\le\alpha$, let $z=(z_1,0,\dots,0)\in\B$ and $\tau=1-|z_1|^2$. For indices $i,j,k\ge2$, with $i\ne j$ wherever both appear,
\[
\begin{array}{ll}
|S^1_{11}F(z)|\le\sqrt{n+1}\,\alpha\,\tau^{-1}, & |S^k_{11}F(z)|\le\sqrt{n+1}\,\alpha\,\tau^{-3/2},\\[4pt]
|S^1_{1j}F(z)|\le2\sqrt{n+1}\,\alpha\,\tau^{-1/2}, & |S^k_{1j}F(z)|\le2\sqrt{n+1}\,\alpha\,\tau^{-1},\\[4pt]
|S^1_{ii}F(z)|\le\sqrt{n+1}\,\alpha, & |S^k_{ii}F(z)|\le\sqrt{n+1}\,\alpha\,\tau^{-1/2},\\[4pt]
|S^1_{ij}F(z)|\le2\sqrt{n+1}\,\alpha, & |S^k_{ij}F(z)|\le2\sqrt{n+1}\,\alpha\,\tau^{-1/2}.
\end{array}
\]
\end{lem}

The bounds are obtained by evaluating $S_F(z)$ on suitably normalized basis vectors, and by polarization for the mixed terms. The rule that summarizes them is simple. Each radial lower index contributes $\tau^{-1}$, each tangential lower index contributes $\tau^{-1/2}$, and the upper index gives back $\tau$ or $\tau^{1/2}$ according to whether it is radial or tangential. This is the anisotropy of the Bergman metric. By invariance under unitary rotations, the lemma describes $S_F$ at any point.

Combining Lemma~\ref{lem:cotas} with \eqref{eq:S0ii}--\eqref{eq:S0ij} and with Cauchy estimates for the derivatives, one also controls the zero order part, and obtains
\begin{equation}\label{eq:cotasS0}
|S^0_{11}F(z_1,0,\dots,0)|\le\frac{C(n,\alpha)}{(1-|z_1|^2)^2},\qquad|S^0_{1j}F(z_1,0,\dots,0)|\le\frac{K(n,\alpha)}{(1-|z_1|^2)^{3/2}}\quad(j\ge2),
\end{equation}
with explicit constants that tend to zero with $\alpha$; for example,
\begin{equation}\label{eq:Cnalpha}
C(n,\alpha)=\Bigl(4n^2+2n-2+\frac{n+1}{n-1}\Bigr)\alpha^2+\Bigl(4\sqrt{n+1}+\frac{8\sqrt{n+1}}{n-1}\Bigr)\alpha\ \le\ 6n^2\alpha^2+16\sqrt n\,\alpha .
\end{equation}
Note the weight $(1-|z_1|^2)^{-2}$ in the first bound, which is the same as that of the classical Schwarzian, as befits a coefficient of order zero. With \eqref{eq:cotasS0}, the restriction of a solution of \eqref{eq:sistema} to a diameter satisfies a second order ordinary differential inequality, and a Sturm comparison shows that the solution normalized by $u(0)=1$, $\nabla u(0)=0$ does not vanish in a ball of radius $r(n,\alpha)>0$. Compactness follows from this.

\begin{thm}[\cite{H06}]\label{thm:normal}
For each $\alpha<\infty$ the family $\Fa$ is normal and compact in the topology of uniform convergence on compact sets.
\end{thm}

This theorem marks a contrast with the plane that is worth emphasizing. In the disk, every univalent function has $\|Sf\|\le6$, so that the entire class $S$ lives inside a family of bounded Schwarzian. In $\C^n$ this is impossible, because the normalized univalent mappings do not form a normal family. A bound for $\|S_F\|$ is therefore a genuine restriction, and there is no Kraus--Nehari theorem.

\medskip

\noindent\textbf{The norm order.} The next step is to estimate the second order coefficients.

\begin{thm}[\cite{H06}; \cite{CH13}, (2.10)]\label{thm:ordennorma}
\[
\|\ord\|\,\Fa\le\frac{\sqrt{n+1}}{2}\,\alpha+\frac{2}{n+1}\,\ord\Fa .
\]
\end{thm}

\begin{proof}
Let $F\in\Fa$ and $|v|=1$. At the origin $J_F(0)=1$ and $P_F(0)=D^2F(0)=2A_2$, so that \eqref{eq:SP} reads
\[
2A_2(v,v)=S_F(0)(v,v)+\frac{2}{n+1}\,\bil{\nabla J_F(0)}{v}\,v .
\]
The Bergman metric at the origin is $(n+1)$ times the Euclidean metric. Therefore $v/\sqrt{n+1}$ is a unit vector for it, and the hypothesis $\|S_F(0)\|\le\alpha$ says that $|S_F(0)(v,v)|\le\sqrt{n+1}\,\alpha$. Consequently
\[
|A_2(v,v)|\le\frac{\sqrt{n+1}}{2}\,\alpha+\frac{1}{n+1}\,|\nabla J_F(0)|,
\]
and $|\nabla J_F(0)|\le2\,\ord\Fa$ by \eqref{eq:orden}.
\end{proof}

The theorem reduces the norm order to the order, and we have stated it in the form in which it appears in \cite{CH13}. In \cite{H06} there is in addition an example in dimension two which, supported by a numerical verification of the norm, gives $\|\ord\|\,\Fa\ge1+\tfrac{\sqrt3}{2}\alpha$; the linear growth in $\alpha$ would then be the correct one.

The identity in the proof can also be read in the opposite direction, and it then bounds the Schwarzian in terms of the two orders,
\[
\sqrt{n+1}\,\|S_F(0)\|\le2\,\|A_2(F)\|+\frac{2}{n+1}\,|\nabla J_F(0)| .
\]
Applied to the family of normalized convex mappings of $\mathbb{B}^2$, which is linearly invariant, has norm order $1$ \cite[Theorem~10.5.1]{GK} and has order at most $1.761$ by a theorem of Barnard, FitzGerald and Gong \cite{BFG}, it gives for every convex mapping $F$ of $\mathbb{B}^2$
\[
\|S_F\|\le\frac{2}{\sqrt3}+\frac{4}{3\sqrt3}\,(1.761)<2.52 .
\]
This is Corollary~3.12 of \cite{H06}, with the constant that results from the preceding theorem. We shall return to convex mappings in Section~\ref{sec:convexas}.

\medskip

\noindent\textbf{The order.} It remains to estimate $\ord\Fa$, and here one can follow closely the variational argument with which Pommerenke computed the order in one variable. Let
\[
A_\alpha=\sup_{F\in\Fa}|\nabla J_F(0)|=2\,\ord\Fa .
\]
The idea is simple. By compactness there is an extremal mapping $F_0$. Its Koebe transforms remain in $\Fa$, and therefore cannot have a larger value of $|\nabla J(0)|$. Writing this condition to first order in the parameter of the transform produces an equation that $F_0$ must satisfy at the origin. We need two lemmas.

\begin{lem}\label{lem:expansion}
Let $F$ be locally biholomorphic near the origin, with $J_F(0)=1$ and $\nabla J_F(0)=\Lambda=(\lambda_1,\dots,\lambda_n)$. Then
\[
\frac{\nabla J_F}{J_F}(\zeta)=\Lambda+A\,\zeta+O(|\zeta|^2),\qquad\zeta\to0,
\]
where $A=(A_{ij})$ is the symmetric matrix
\begin{equation}\label{eq:matrizA}
A_{ij}=\sum_{k=1}^nS^k_{ij}F(0)\,\lambda_k-(n+1)\,S^0_{ij}F(0)+\frac{\lambda_i\lambda_j}{n+1}.
\end{equation}
\end{lem}

\begin{proof}
Let $u_0=J_F^{-1/(n+1)}$, so that $\partial_i\log J_F=-(n+1)\,\partial_iu_0/u_0$. Since $u_0(0)=1$,
\[
\partial_j\Bigl(\frac{\partial_iu_0}{u_0}\Bigr)(0)=\partial_{ij}u_0(0)-\partial_iu_0(0)\,\partial_ju_0(0).
\]
Now $\nabla u_0(0)=-\Lambda/(n+1)$, and $u_0$ is a solution of the system \eqref{eq:sistema} associated with $F$, hence
\[
\partial_{ij}u_0(0)=-\frac{1}{n+1}\sum_kS^k_{ij}F(0)\,\lambda_k+S^0_{ij}F(0).
\]
Multiplying by $-(n+1)$ we obtain \eqref{eq:matrizA}.
\end{proof}

\begin{lem}\label{lem:automorfismo}
For each $\zeta\in\B$ there exists $\sigma\in\Aut(\B)$ with $\sigma(0)=\zeta$ such that
\[
D\sigma(0)=\Id+O(|\zeta|^2),\qquad\frac{\nabla J_\sigma}{J_\sigma}(0)=-(n+1)\,\overline{\zeta}.
\]
\end{lem}

\begin{proof}
If $\zeta=(\zeta_1,0,\dots,0)$ we take
\begin{equation}\label{eq:sigma}
\sigma(z)=\Bigl(\frac{z_1+\zeta_1}{1+\overline{\zeta_1}z_1},\ \frac{\sqrt{1-|\zeta_1|^2}\;z_2}{1+\overline{\zeta_1}z_1},\ \dots,\ \frac{\sqrt{1-|\zeta_1|^2}\;z_n}{1+\overline{\zeta_1}z_1}\Bigr),
\end{equation}
for which
\[
J_\sigma(z)=\frac{(1-|\zeta_1|^2)^{\frac{n+1}{2}}}{(1+\overline{\zeta_1}z_1)^{n+1}},\qquad D\sigma(0)=\operatorname{diag}\bigl(1-|\zeta_1|^2,\sqrt{1-|\zeta_1|^2},\dots,\sqrt{1-|\zeta_1|^2}\bigr).
\]
The general case is obtained by conjugating with a unitary rotation.
\end{proof}

\begin{thm}[\cite{CH13}]\label{thm:caracteristica}
Let $F_0\in\Fa$ be extremal for $A_\alpha$, and suppose, after a rotation, that $\nabla J_{F_0}(0)=(\lambda,0,\dots,0)$ with $\lambda=A_\alpha>0$. Then, at the origin,
\begin{align}
&\lambda^2+(n+1)\,S^1_{11}F_0\,\lambda-(n+1)^2S^0_{11}F_0-(n+1)^2=0,\label{eq:carac1}\\
&S^1_{1j}F_0\,\lambda-(n+1)\,S^0_{1j}F_0=0,\qquad j=2,\dots,n.\label{eq:carac2}
\end{align}
\end{thm}

\begin{proof}
Let $\zeta\in\B$ be small, $\sigma$ as in Lemma~\ref{lem:automorfismo} and $G=\Lambda_\sigma(F_0)\in\Fa$. From \eqref{eq:koebe},
\[
J_G(z)=J_\sigma(0)^{-1}J_{F_0}(\zeta)^{-1}\,J_{F_0}(\sigma(z))\,J_\sigma(z),
\]
and differentiating logarithmically at $z=0$, with the two lemmas,
\[
\nabla J_G(0)=\frac{\nabla J_{F_0}}{J_{F_0}}(\zeta)\,D\sigma(0)+\frac{\nabla J_\sigma}{J_\sigma}(0)=\Lambda+A\,\zeta-(n+1)\,\overline\zeta+O(|\zeta|^2),
\]
with $\Lambda=(\lambda,0,\dots,0)$. Only the first component contributes to the modulus to first order, and we find
\[
|\nabla J_G(0)|^2=\lambda^2+2\lambda\,\RE\Bigl[\sum_{j=1}^nA_{1j}\zeta_j-(n+1)\,\zeta_1\Bigr]+O(|\zeta|^2).
\]
Since $F_0$ is extremal, the left-hand side does not exceed $\lambda^2$. The real part of a linear form in $\zeta$ that is less than or equal to $O(|\zeta|^2)$ for all small $\zeta$ vanishes identically, and therefore
\[
A_{11}=n+1,\qquad A_{1j}=0\quad(j\ge2).
\]
With $\Lambda=(\lambda,0,\dots,0)$, formula \eqref{eq:matrizA} gives $A_{1j}=S^1_{1j}\lambda-(n+1)S^0_{1j}+\delta_{1j}\lambda^2/(n+1)$, and the two preceding relations are \eqref{eq:carac1} and \eqref{eq:carac2}.
\end{proof}

Equation \eqref{eq:carac1} is the analogue of Pommerenke's characteristic equation. It involves derivatives of $F_0$ up to order three, through $S^0_{11}$, and it is quadratic in $\lambda$.

\begin{thm}[\cite{CH13}]\label{thm:ordensup}
\[
\ord\Fa\le\frac{n+1}{2}\Bigl(\frac{\sqrt{n+1}}{2}\,\alpha+\sqrt{1+\frac{n+1}{4}\,\alpha^2+C(n,\alpha)}\Bigr),
\]
with $C(n,\alpha)$ as in \eqref{eq:Cnalpha}.
\end{thm}

\begin{proof}
Completing the square in \eqref{eq:carac1},
\[
\Bigl(\lambda+\tfrac12(n+1)S^1_{11}\Bigr)^2=(n+1)^2\Bigl(1+\tfrac14(S^1_{11})^2+S^0_{11}\Bigr),
\]
and it suffices to use $|S^1_{11}F_0(0)|\le\sqrt{n+1}\,\alpha$ and $|S^0_{11}F_0(0)|\le C(n,\alpha)$, which are Lemma~\ref{lem:cotas} and \eqref{eq:cotasS0} at the origin.
\end{proof}

For $\alpha=0$ the bound equals $(n+1)/2$, which is the exact value, since for $T_a$ we have $|\nabla J_{T_a}(0)|=(n+1)|a|$ with $|a|\le1$. Together with Theorem~\ref{thm:ordennorma} we obtain
\[
\|\ord\|\,\Fa\le\sqrt{n+1}\,\alpha+\sqrt{1+\frac{n+1}{4}\,\alpha^2+C(n,\alpha)} .
\]

To bound the order from below examples are needed, and Theorem~\ref{thm:oda} allows us to construct them without writing down mappings. By Corollary~\ref{cor:soluciones}, it suffices to exhibit a completely integrable system in canonical form and a solution $u$ that does not vanish in $\B$, with $u(0)=1$ and $|\nabla u(0)|$ large. The simplest choice is to take constant coefficients, because then the norm is computed at a single point.

\begin{lem}[\cite{CH13}]\label{lem:constantes}
If $F$ is locally biholomorphic in $\B$ and all the $S^k_{ij}F$, $k\ge1$, are constant, then $\|S_F\|=\|S_F(0)\|$.
\end{lem}

\begin{proof}
From \eqref{eq:bergmanbola} and the Cauchy--Schwarz inequality, for every vector $w$
\[
\frac{(n+1)\,|w|^2}{1-|z|^2}\le\|w\|_z^2\le\frac{(n+1)\,|w|^2}{(1-|z|^2)^2}.
\]
If $\|v\|_z=1$, the first inequality gives $|v|^2\le(1-|z|^2)/(n+1)$, and the second, applied to $w=S_F(z)(v,v)$, which does not depend on $z$, gives
\[
\|S_F(z)(v,v)\|_z\le\frac{\sqrt{n+1}\,|v|^2}{1-|z|^2}\,\Bigl|S_F\Bigl(\frac{v}{|v|},\frac{v}{|v|}\Bigr)\Bigr|\le\frac{1}{\sqrt{n+1}}\,\sup_{|e|=1}|S_F(e,e)|=\|S_F(0)\| .\qedhere
\]
\end{proof}

\begin{thm}[\cite{CH13}]\label{thm:ordeninf}
For $n=2$ we have $\ord\Fa\ge\tfrac94\alpha$, and for $n>2$,
\[
\ord\Fa\ge\frac{(n+1)^{3/2}}{2(n-1)}\,\alpha .
\]
\end{thm}

\begin{proof}
Let $\kappa=\sqrt{n+1}\,\alpha/(n-1)$ and consider the system with constant coefficients
\[
S^1_{11}=-\sqrt{n+1}\,\alpha,\qquad S^k_{1k}=\kappa\quad(k\ge2),
\]
with the remaining $S^k_{ij}$ equal to zero. It satisfies \eqref{eq:canonica}, and formulas \eqref{eq:S0ii}--\eqref{eq:S0ij} give $S^0_{11}=n\kappa^2$ and $S^0_{ij}=0$ in the remaining cases. The system reads
\[
u_{11}=-\sqrt{n+1}\,\alpha\,u_1+n\kappa^2u,\qquad u_{1j}=\kappa\,u_j\quad(j\ge2),\qquad u_{ij}=0\quad(i,j\ge2),
\]
and it can be integrated by hand. Its solutions are the combinations of $e^{\kappa z_1}$, $e^{-n\kappa z_1}$ and $e^{\kappa z_1}z_j$, $j\ge2$. There are $n+1$ of them, and the system is completely integrable. The solution $u=e^{\kappa z_1}$ does not vanish, and has $u(0)=1$, $\nabla u(0)=(\kappa,0,\dots,0)$. By Theorem~\ref{thm:oda}(b), applied to a basis of solutions whose first element is $u$, there is a mapping $G$ locally biholomorphic in $\B$, which we may normalize, with $J_G=u^{-(n+1)}$ and, therefore,
\[
|\nabla J_G(0)|=(n+1)\,\kappa=\frac{(n+1)^{3/2}}{n-1}\,\alpha .
\]
It remains to compute the norm. At the origin $S(v,v)=\bigl(-\sqrt{n+1}\,\alpha\,v_1^2,\,2\kappa v_1v_2,\dots,2\kappa v_1v_n\bigr)$, and maximizing over the vectors of Bergman norm one we find $\|S(0)\|=\alpha$ if $n>2$ and $\|S(0)\|=\tfrac{2}{\sqrt3}\alpha$ if $n=2$. By Lemma~\ref{lem:constantes} this is the value of $\|S_G\|$. For $n>2$ this gives $G\in\Fa$ and the stated bound. For $n=2$ it gives $G\in\mathcal F_\beta$ with $\beta=\tfrac{2}{\sqrt3}\alpha$ and $|\nabla J_G(0)|=3\sqrt3\,\alpha=\tfrac92\beta$.
\end{proof}

In summary, the order of $\Fa$ grows linearly with $\alpha$, between constants that differ by powers of $n$. The family $\Fa$ behaves, up to this point, like its one-variable model. The main point remains, which is to decide whether its elements are univalent.

\section{Univalence criteria in the ball}\label{sec:univalencia}

Recall the three steps in the proof of Theorem~\ref{thm:nehari}, which are the zero lemma, the reduction to a diameter and the Sturm comparison. The first has an exact analogue.

\begin{lem}[\cite{H07}]\label{lem:ceroscomunes}
Let $F:\Omega\to\C^n$ be locally biholomorphic in a simply connected domain. Then $F$ is univalent if and only if there do not exist $n$ linearly independent solutions of the system \eqref{eq:sistema} associated with $F$ that have two common zeros in $\Omega$.
\end{lem}

\begin{proof}
Let $V$ be the space of solutions, of dimension $n+1$, and for $a\in\Omega$ let $V_a=\{u\in V:\ u(a)=0\}$. Since $u_0=J_F^{-1/(n+1)}$ does not vanish, $V_a$ is a hyperplane of $V$. With the basis \eqref{eq:base}, $F(a)=F(b)$ if and only if the vectors $(u_0(a),\dots,u_n(a))$ and $(u_0(b),\dots,u_n(b))$ are proportional, that is, if and only if $V_a=V_b$. And $V_a=V_b$ with $a\ne b$ means precisely that there are $n$ independent solutions that vanish at $a$ and at $b$.
\end{proof}

For $n=1$ we recover Lemma~\ref{lem:ceros1}. For $n\ge2$ the new difficulty is in plain view. It is no longer enough to prevent one solution from having two zeros; one must prevent an entire hyperplane of solutions from sharing them.

The second step can also be carried out. If $a$ and $b$ are the common zeros, we restrict the solutions to the segment joining them, or, in the ball and after an automorphism, to a diameter. What results is a system of second order ordinary differential equations with a term of order zero, which comes from $S^0_F$, and a term of \emph{first order}, which comes from $S_F$ and which does not exist in one variable. The comparison must take it into account, and in \cite{H07} this is achieved with a nonhomogeneous variant of Sturm's lemma.

With uniform Euclidean bounds for the coefficients one obtains in this way a criterion valid in any bounded convex domain \cite{H07}, in which the first order term is paid for with an explicit loss with respect to the condition $\delta\le(\pi/\diam\Omega)^2$ that the equation $u''+\delta u=0$ would give. The invariant statement in the ball is the following.

\begin{thm}[\cite{H07}]\label{thm:H07}
Let $F:\B\to\C^n$ be locally biholomorphic and such that
\begin{enumerate}
\item[(i)] $\|S_F(z)\|\le\alpha$ for all $z\in\B$;
\item[(ii)] $\displaystyle\sup_{\|a\|_z=\|b\|_z=1}|a^tS^0F(z)b|\le\frac{k}{1-|z|^2}$ for all $z\in\B$.
\end{enumerate}
If
\[
2\sqrt{n(n+1)}\,\bigl(e^{\pi/2}+2n+1\bigr)\alpha+(n+1)k\le1,
\]
then $F$ is univalent in $\B$.
\end{thm}

The theorem has a defect, and it is instructive to understand where it comes from. For $n\ge2$ hypothesis (ii) should be superfluous, because the $S^0_{ij}F$ are determined by the $S^k_{ij}F$ and are bounded in terms of $\alpha$ by \eqref{eq:cotasS0}. But those bounds carry the weight $(1-|z|^2)^{-2}$, and at that level the Sturm comparison with a first order term does not close. In one variable the critical case is resolved because the extremal solution $\sqrt{1-x^2}$ is known, and here nothing of the kind is known. Whether the condition $\|S_F\|\le\alpha_0$ alone suffices for univalence was left as an open question in \cite{H07}. The answer is affirmative, but it requires replacing the third step of Nehari's proof by another idea; this is the subject of the next section.

\medskip

\noindent\textbf{A criterion via Loewner chains.} Before that, we record a criterion of a different nature. Nehari's theorem admits a proof by Loewner chains, as does Becker's, and that proof can indeed be transcribed to $\C^n$.

\begin{thm}[\cite{H22}]\label{thm:loewner}
Let $F\in\mathcal{LS}(\B)$ and let $p=\nabla\log J_F(z)$. If for all $z\in\B$
\begin{equation}\label{eq:critloewner}
\frac{\|S_F(z)(z,\cdot)\|}{1-|z|^2}+\Bigl\|\frac{\bil{p}{z}\,S_F(z)(z,\cdot)-\bil{p}{S_F(z)(z,z)}\,\Id}{n+1}+S^0_F(z)(z,z)\,\Id\Bigr\|\le\frac{1}{(1-|z|^2)^2},
\end{equation}
then $F$ is univalent in $\B$.
\end{thm}

\begin{proof}[Idea of the proof]
Let $u_0,u_1,\dots,u_n$ be the basis \eqref{eq:base} and $u=(u_1,\dots,u_n)$. We define
\[
f(z,t)=\frac{u(e^{-t}z)+(e^{2t}-1)\,Du(e^{-t}z)(e^{-t}z)}{u_0(e^{-t}z)+(e^{2t}-1)\bil{\nabla u_0(e^{-t}z)}{e^{-t}z}},\qquad t\ge0,
\]
which is the transcription to $\C^n$ of the chain used to prove Nehari's criterion, and satisfies $f(z,0)=F(z)$. Upon differentiating, the second derivatives of $u$ and of $u_0$ appear, and they are eliminated with \eqref{eq:sistemacompacto}. The result is that $f$ satisfies \eqref{eq:loewnerpde} with
\[
h(z,t)=\bigl[\Id-E(z,t)\bigr]^{-1}\bigl[\Id+E(z,t)\bigr]z,
\]
where $E(z,t)$ is a linear operator built from $S_F$, $S^0_F$ and $\nabla\log u_0$ evaluated at $e^{-t}z$. Hypothesis \eqref{eq:critloewner}, together with $1-e^{-2t}<1-|e^{-t}z|^2$, gives $\|E(z,t)\|<1$, and then $\RE(h(z,t)\cdot z)>0$. The theorem quoted in Section~\ref{sec:prelim} completes the proof.
\end{proof}

For $n=1$ we have $S_f\equiv0$ and $S^0_f=-\tfrac12Sf$, and \eqref{eq:critloewner} reduces to $|z|^2|Sf(z)|\le2(1-|z|^2)^{-2}$, which is a consequence of Nehari's condition, so that the theorem recovers the classical criterion. For $n\ge2$ the condition involves $S^0_F$ and the gradient of the Jacobian, and it is not a condition on $\|S_F\|$ alone. There is one case in which it simplifies.

\begin{cor}[\cite{H22}]\label{cor:jacconst}
If $F\in\mathcal{LS}(\B)$ has constant Jacobian and $(1-|z|^2)\|S_F(z)(z,\cdot)\|\le1$ in $\B$, then $F$ is univalent.
\end{cor}

Indeed, if $J_F$ is constant then $u_0$ is constant and \eqref{eq:sistema} forces $S^0_{ij}F\equiv0$. In this case $S_F=P_F$ by \eqref{eq:SP}, and the corollary coincides with Theorem~\ref{thm:pfaltzgraff}, as it should. What is new in Theorem~\ref{thm:loewner} lies in the case of nonconstant Jacobian, where the Schwarzian subtracts from $P_F$ the part that comes from a Möbius transformation.
\section{The Ahlfors--Weill extension}\label{sec:AW}

We come to the central result of the chapter, due to Chuaqui and the author \cite{CH15}. In the ball, a sufficiently small bound on $\|S_F\|$ implies univalence, with no additional hypotheses, and with it an extension of Ahlfors--Weill type. The idea that replaces the Sturm comparison is geometric, and it has its origin in the work of Chuaqui and Osgood in one variable \cite{CO94}. If $f$ is analytic in $\D$ and $u=(f')^{-1/2}$, the function $|u(z)|/\sqrt{1-|z|^2}$ is the square root of the hyperbolic density of $f(\D)$, evaluated at $f(z)$. When $\|Sf\|\le2$ it is convex along hyperbolic geodesics, and from this convexity one deduces both Nehari's theorem and the properties of the extension.

In $\B$, given $F\in\Fa$ and a solution $u$ of the associated system \eqref{eq:sistema}, we consider the \emph{density}
\begin{equation}\label{eq:densidad}
\rho_u(z)=\frac{|u(z)|}{\sqrt{1-|z|^2}},
\end{equation}
and we write $\rho_F$ when $u=J_F^{-1/(n+1)}$. The density is compatible with the automorphisms.

\begin{lem}\label{lem:rhoinv}
If $\sigma\in\Aut(\B)$ and $F$ is locally biholomorphic in $\B$, then $\rho_{F\circ\sigma}=\rho_F\circ\sigma$.
\end{lem}

\begin{proof}
This is a consequence of $J_{F\circ\sigma}=(J_F\circ\sigma)J_\sigma$ and of the identity
\[
|J_\sigma(z)|=\Bigl(\frac{1-|\sigma(z)|^2}{1-|z|^2}\Bigr)^{\frac{n+1}{2}},
\]
which for $\sigma$ of the form \eqref{eq:sigma} is verified directly and for unitary rotations is evident.
\end{proof}

The first step is a differential inequality along a diameter. Here the first-order term appears, and the way to control it is the order of the family. Let
\begin{equation}\label{eq:delta}
\delta(n,\alpha)=\sqrt{n+1}\,(A_\alpha+1)\,\alpha+C(n,\alpha),
\end{equation}
with $A_\alpha=2\ord\Fa$ and $C(n,\alpha)$ as in \eqref{eq:Cnalpha}.

\begin{lem}[\cite{CH15}]\label{lem:radial}
Let $F\in\Fa$ and let $u$ be a solution of the associated system that does not vanish in $\B$. Then $\varphi(x)=|u(x,0,\dots,0)|$ satisfies
\[
\varphi''(x)+\frac{\delta(n,\alpha)}{(1-x^2)^2}\,\varphi(x)\ge0,\qquad-1<x<1 .
\]
\end{lem}

\begin{proof}
By Corollary~\ref{cor:soluciones}, $u=J_G^{-1/(n+1)}$ with $G=T\circ F$ holomorphic in $\B$, and $\|S_G\|=\|S_F\|\le\alpha$. After an affine normalization, which only multiplies $u$ by a constant, we may assume that $G\in\Fa$; we write $F$ instead of $G$.

Fix $\zeta=(x,0,\dots,0)$, let $\sigma$ be the automorphism \eqref{eq:sigma} and $H=\Lambda_\sigma(F)\in\Fa$. Then $v=J_H^{-1/(n+1)}$ is a constant multiple of $(u\circ\sigma)\,w$, with $w=J_\sigma^{-1/(n+1)}$, and differentiating logarithmically at the origin,
\[
\frac{\nabla u}{u}(\zeta)\,D\sigma(0)=\frac{\nabla v}{v}(0)-\frac{\nabla w}{w}(0).
\]
The first term on the right is $-\nabla J_H(0)/(n+1)$, of norm at most $A_\alpha$, and the second is $\overline\zeta$ by Lemma~\ref{lem:automorfismo}. Recalling the form of $D\sigma(0)$,
\begin{equation}\label{eq:gradiente}
(1-x^2)^2\,|\partial_1u(\zeta)|^2+(1-x^2)\sum_{k=2}^n|\partial_ku(\zeta)|^2\le(A_\alpha+1)^2\,|u(\zeta)|^2 .
\end{equation}
On the other hand, the vector $(1-x^2)e_1/\sqrt{n+1}$ is a unit vector for the Bergman metric at $\zeta$, and the hypothesis $\|S_F(\zeta)\|\le\alpha$ applied to it gives
\[
(1-x^2)^2\,|S^1_{11}F(\zeta)|^2+(1-x^2)^3\sum_{k=2}^n|S^k_{11}F(\zeta)|^2\le(n+1)\,\alpha^2 .
\]
From these two inequalities and the Cauchy--Schwarz inequality,
\[
\Bigl|\sum_{k=1}^nS^k_{11}F(\zeta)\,\partial_ku(\zeta)\Bigr|\le\frac{\sqrt{n+1}\,(A_\alpha+1)\,\alpha}{(1-x^2)^2}\,|u(\zeta)| .
\]
Finally, as in the proof of Theorem~\ref{thm:nehari},
\[
\varphi''\ \ge\ \frac{\RE(\overline u\,\partial_{11}u)}{|u|},\qquad\partial_{11}u=\sum_{k=1}^nS^k_{11}F\,\partial_ku+S^0_{11}F\,u .
\]
The first summand of $\partial_{11}u$ has just been estimated, and the second is controlled by \eqref{eq:cotasS0}.
\end{proof}

Compare with one variable. There is no first-order term there, $A_\alpha$ does not enter, and the inequality is $\varphi''+\tfrac12|Sf|\varphi\ge0$. The price of the first-order term is that the bound on the Schwarzian appears multiplied by the order of the family.

\begin{thm}[\cite{CH15}]\label{thm:convexidad}
Let $F\in\Fa$ with $\delta(n,\alpha)<1$, and let $u$ be a solution of the associated system that does not vanish in $\B$. Then
\begin{equation}\label{eq:hess}
\Hess_g\rho_u\ge\beta^2\rho_u\,g,\qquad\beta^2=\frac{1-\delta(n,\alpha)}{n+1},
\end{equation}
where $g$ is the Bergman metric and $\Hess_g$ is the Riemannian Hessian. In particular, $\rho_u$ is strictly convex along the geodesics of $g$, and it has at most one critical point in $\B$, which is then its absolute minimum.
\end{thm}

\begin{proof}
We must prove that $(\rho_u\circ\gamma)''\ge\beta^2\rho_u\circ\gamma$ for every geodesic $\gamma$ parametrized by arc length. Consider first the diameter $\gamma(t)=(x(t),0,\dots,0)$, for which $x'=(1-x^2)/\sqrt{n+1}$. With $\rho=\varphi\,(1-x^2)^{-1/2}$ and $\varphi$ as in Lemma~\ref{lem:radial},
\[
\frac{d\rho}{dt}=\frac{1}{\sqrt{n+1}}\Bigl[\varphi'\,(1-x^2)^{1/2}+x\,\varphi\,(1-x^2)^{-1/2}\Bigr],
\]
and differentiating once more the terms in $\varphi'$ cancel, and we are left with
\[
\frac{d^2\rho}{dt^2}=\frac{1-x^2}{n+1}\Bigl[\varphi''\,(1-x^2)^{1/2}+\varphi\,(1-x^2)^{-3/2}\Bigr]\ge\frac{1-\delta}{n+1}\cdot\frac{\varphi}{(1-x^2)^{1/2}}=\beta^2\rho .
\]
An arbitrary geodesic is the image of the diameter under an automorphism $\sigma$. By Lemma~\ref{lem:rhoinv}, $\rho_u\circ\sigma$ is a constant multiple of the density of a solution, which does not vanish either, of the system associated with the Koebe transform $\Lambda_\sigma(F)\in\Fa$, and the preceding argument applies. A function that is strictly convex along geodesics, in a space in which any two points are joined by a geodesic, has at most one critical point.
\end{proof}

From now on we fix $\alpha_0=\alpha_0(n)>0$ such that $\delta(n,\alpha_0)<1$, and we denote by $\beta$ the corresponding constant in \eqref{eq:hess}. Since $A_\alpha$ is of the order of $n$ by Theorem~\ref{thm:ordensup}, the condition forces $\alpha_0$ to be of order $n^{-3/2}$.

The theorem has a limitation that should be pointed out. It applies only to solutions that do not vanish in the whole ball, because the estimate \eqref{eq:gradiente} uses the order of the family, and this is lost when the mapping $G=T\circ F$ has poles. To treat solutions with zeros one works in smaller Bergman balls. Let $B_r=\{|z|<s\}$ be the Bergman ball with center $0$ and radius $r$, with $s=s(r)$ its Euclidean radius, and for a solution $u$ that does not vanish in $B_r$ let
\[
\rho_{u,B_r}(z)=\frac{|u(z)|}{\sqrt{s^2-|z|^2}} .
\]

\begin{lem}[\cite{CH15}]\label{lem:subbola}
Let $F\in\mathcal F_{\alpha_0}$ and let $u$ be a solution of the associated system that does not vanish in $B_r$. Then $\rho_{u,B_r}$ is convex with respect to the Bergman metric of the ball $B_r$.
\end{lem}

The proof consists of rescaling. The norm of $S_F$ with respect to the Bergman metric of $B_r$ does not exceed $\|S_F\|$, so that $z\mapsto G(sz)$, with $u=J_G^{-1/(n+1)}$, satisfies in $\B$ the hypotheses of Theorem~\ref{thm:convexidad}.

\begin{thm}[\cite{CH15}]\label{thm:nocerosbola}
Let $F\in\mathcal F_{\alpha_0}$, and let $u$ be the solution of the associated system with $u(0)=1$, $\nabla u(0)=0$. Then $u$ does not vanish in $\B$, and
\[
|u(z)|\ge\sqrt{1-|z|^2}\,\cosh\bigl(\beta\,d(0,z)\bigr),\qquad d(0,z)=\frac{\sqrt{n+1}}{2}\log\frac{1+|z|}{1-|z|},
\]
where $d$ is the Bergman distance.
\end{thm}

\begin{proof}
Let $s\le1$ be the supremum of the Euclidean radii of the balls centered at the origin in which $u$ does not vanish, and suppose $s<1$. Then $u(z_1)=0$ for some $z_1$ with $|z_1|=s$. By Lemma~\ref{lem:subbola}, $\rho_{u,B_r}$ is convex in $B_r=\{|z|<s\}$, and since $\nabla u(0)=0$, the origin is a critical point and therefore a minimum. Thus
\[
|u(z)|\ge\frac1s\sqrt{s^2-|z|^2},\qquad|z|<s .
\]
As $z\to z_1$ along the radius, the left-hand side is $O(s-|z|)$ and the right-hand side is comparable to $\sqrt{s-|z|}$, which is absurd. Hence $s=1$. Now Theorem~\ref{thm:convexidad} applies to $u$. Along each geodesic through the origin, $\rho=\rho_u$ satisfies $\rho''\ge\beta^2\rho$, $\rho(0)=1$, $\rho'(0)=0$, and by comparison $\rho\ge\cosh(\beta t)$.
\end{proof}

Let us read the theorem geometrically. By Corollary~\ref{cor:soluciones} and \eqref{eq:Ta}, the normalized solution is $u=J_F^{-1/(n+1)}\bigl(1+\bil{a}{F}\bigr)$ with $(n+1)\,a=\nabla J_F(0)$, and the fact that it does not vanish means that
\begin{equation}\label{eq:hiperplano}
F(\B)\cap H(0)=\emptyset,\qquad H(0)=\{w\in\C^n:\ 1+\bil{a}{w}=0\}.
\end{equation}
For $n=1$ the hyperplane $H(0)$ is the point $-1/a_2$, the value at $\infty$ of the Ahlfors--Weill extension, as we saw in Section~\ref{sec:modelo}. The same holds at any other point. Given $z_0\in\B$, we apply the theorem to a Koebe transform $\Lambda_\sigma(F)$ with $\sigma(0)=z_0$ and return to $F$ by Lemma~\ref{lem:rhoinv}. We obtain a solution $u_{z_0}$ of the system associated with $F$ that does not vanish in $\B$ and whose density has a critical point at $z_0$. Its zero set, viewed in the image, is a complex hyperplane $H(z_0)$ disjoint from $F(\B)$. For $n=1$ it is the point $E_f(z_0)$ of formula \eqref{eq:AW1}.

To prove univalence we need Bergman balls that are not centered at the origin, and Lemma~\ref{lem:subbola} extends to them with no difficulty. Let $B=B(z_0,r)$ be the Bergman ball of center $z_0\in\B$ and radius $r$, let $\psi$ be an automorphism of $\B$ taking $z_0$ to the origin, and hence $B$ onto $B_r$, and for a solution $u$ that does not vanish in $B$ let
\[
\rho_{u,B}(z)=\frac{|u(z)|\,|J_\psi(z)|^{\frac{1}{n+1}}}{\sqrt{s^2-|\psi(z)|^2}} .
\]
If $u=J_G^{-1/(n+1)}$ and $v=J_{G\circ\psi^{-1}}^{-1/(n+1)}$, then $\rho_{u,B}=\rho_{v,B_r}\circ\psi$, and since $\|S_{G\circ\psi^{-1}}\|=\|S_G\|$, Lemma~\ref{lem:subbola} applied to $v$ shows that $\rho_{u,B}$ is convex with respect to the Bergman metric of $B$.

\begin{thm}[\cite{CH15}]\label{thm:univCH15}
Every $F\in\mathcal F_{\alpha_0}$ is univalent in $\B$.
\end{thm}

\begin{proof}
Suppose that $F$ is not injective. By Lemma~\ref{lem:ceroscomunes} there are $n$ linearly independent solutions $u_1,\dots,u_n$ of the associated system that vanish at two distinct points, and by the linear invariance of the family we may assume that these points are the origin and $p=(a,0,\dots,0)$ with $0<a<1$. The gradients $\nabla u_i(0)$ are linearly independent, because a solution that vanishes at the origin together with its gradient is identically zero. Hence there is a linear combination $u$ with $\nabla u(0)=(1,0,\dots,0)$. Its zero set $\Sigma$ is a smooth complex hypersurface that passes through the origin and through $p$, and whose tangent plane at the origin is $\{z_1=0\}$.

For $0<b<1$ let $B^b$ be the Bergman ball with center $(b,0,\dots,0)$ whose boundary passes through the origin. Writing $z=(z_1,z')$, it is the set
\[
B^b=\bigl\{z\in\B:\ |1-bz_1|^2<1-|z|^2\bigr\}=\bigl\{2b\,\RE z_1>b^2|z_1|^2+|z|^2\bigr\}.
\]
These balls increase with $b$, they are all tangent to $\{z_1=0\}$ at the origin, and $p\in B^b$ as soon as $2b>a(1+b^2)$, which happens for $b$ close to $1$.

We compare $\Sigma$ with these balls near the origin. There $\Sigma$ is a graph $z_1=\phi(z')$ with $\phi$ holomorphic, $\phi(0)=0$ and $\nabla\phi(0)=0$, and its extrinsic curvature at the origin, in the direction $(0,v')$, is $|\Hess\phi(0)(v',v')|$. By \eqref{eq:curvaturabergman} this curvature does not exceed $\sqrt{n+1}\,\alpha_0$, which is less than $1$ because $\delta(n,\alpha_0)<1$. Hence there are a neighborhood $U$ of the origin and a constant $c<\tfrac12$ such that $|\phi(z')|\le c\,|z'|^2$ on $\Sigma\cap U$. If a point of $\Sigma\cap U$ were in the closure of some $B^b$, we would have $|z'|^2\le2b\,\RE\phi(z')\le2c\,|z'|^2$, and therefore $z'=0$. Thus, inside $U$, the hypersurface $\Sigma$ meets the closures of all the balls $B^b$ only at the origin. In particular $u$ does not vanish in $B^b$ when $b$ is small.

Let $b_1$ be the supremum of the $b$ such that $u$ does not vanish in $B^b$. If $b_1=1$ we have a contradiction, because $u(p)=0$ and $p$ belongs to $B^b$ for $b$ close to $1$. Suppose then that $b_1<1$ and let $B_1=B^{b_1}$. The solution $u$ does not vanish in $B_1$, and it has zeros in $B^b$ for every $b>b_1$. These zeros lie outside $U$, and a limit point of them as $b\to b_1$ is a point $z_2$ of the boundary of $B_1$, different from the origin, with $u(z_2)=0$.

The density $\rho_{u,B_1}$ is convex with respect to the Bergman metric of $B_1$. The geodesic of this metric that joins the boundary points $0$ and $z_2$ reaches them transversally. Near each of them, $|u|$ is at most proportional to the distance to the point, while the denominator of $\rho_{u,B_1}$ is comparable to the square root of that distance. Hence $\rho_{u,B_1}$ tends to zero at both ends of the geodesic. But a positive convex function on the whole line cannot tend to zero at both ends. This contradiction proves the theorem.
\end{proof}

This is the genuine analogue of Nehari's theorem in the ball, with a single hypothesis, on $\|S_F\|$. The proof uses in an essential way the geometry of the preimages of hyperplanes. What prevents $\Sigma$ from meeting the balls $B^b$ again near the origin is the bound on its curvature given by Theorem~\ref{thm:curvatura}.

\begin{obs}\label{obs:univdirecta}
Once it is known that the solutions $u_{z_0}$ do not vanish, univalence can also be deduced without auxiliary balls. Suppose $F(a)=F(b)$ with $a\ne b$, and let $u_a$, $u_b$ be the solutions constructed before the theorem. By Theorem~\ref{thm:convexidad} the density $\rho_{u_a}$ has its unique minimum at $a$ and $\rho_{u_b}$ has it at $b$, so that
\[
\rho_{u_a}(a)<\rho_{u_a}(b),\qquad\rho_{u_b}(b)<\rho_{u_b}(a).
\]
But $u_a=J_F^{-1/(n+1)}\,l_a(F)$ and $u_b=J_F^{-1/(n+1)}\,l_b(F)$ with $l_a,l_b$ affine functions. The quotient $u_a/u_b=l_a(F)/l_b(F)$ depends only on the value of $F$ and takes the same value at $a$ and at $b$, and then
\[
\frac{\rho_{u_a}(a)}{\rho_{u_b}(a)}=\frac{\rho_{u_a}(b)}{\rho_{u_b}(b)},
\]
which contradicts the two preceding inequalities. The same argument shows that the hyperplanes $H(z_0)$ are distinct for distinct points $z_0$.
\end{obs}

The convexity also gives a distortion theorem. We say that $F\in\mathcal F_{\alpha_0}$ is \emph{normalized} if $\nabla J_F(0)=0$; every $F\in\mathcal F_{\alpha_0}$ is brought to this form by composing with the transformation $T_a$ of \eqref{eq:hiperplano}, which is holomorphic in $F(\B)$.

\begin{thm}[\cite{CH15}]\label{thm:distorsion}
If $F\in\mathcal F_{\alpha_0}$ is normalized, then
\[
|J_F(z)|\le\frac{2^{n+1}}{(1-|z|)^{\mu}},\qquad\mu=\frac{n+1}{2}\bigl(1-\beta\sqrt{n+1}\bigr).
\]
If in addition $F$ is quasiregular, $\|DF(z)\|\le C\,(1-|z|)^{-\mu/n}$, and in particular $F(\B)$ is bounded and $F$ has a continuous extension, in fact a Hölder continuous one, to $\overline{\B}$.
\end{thm}

\begin{proof}
In this case the normalized solution is $J_F^{-1/(n+1)}$, and Theorem~\ref{thm:nocerosbola} gives
\[
\frac{|J_F(z)|^{-\frac{1}{n+1}}}{\sqrt{1-|z|^2}}\ge\cosh\bigl(\beta\,d(0,z)\bigr)\ge\frac12\Bigl(\frac{1}{1-|z|}\Bigr)^{\frac{\beta\sqrt{n+1}}{2}} ,
\]
from which the first bound follows. If $F$ is quasiregular, $\|DF\|^n$ is comparable to $|J_F|$, and since $\mu/n<1$ the resulting bound is integrable along the radii.
\end{proof}

The exponent $\mu$ is smaller than the one that would be obtained from the order of the family, and it tends to zero as $\alpha_0\to0$. The inequality \eqref{eq:hess} with $\beta>0$ gives exponential growth of the density in the Bergman distance, which is more information than a bound for $|\nabla J_F|$.

In $\C^n$ there is no point at infinity to add, nor an evident reflection. What there is instead is the family of hyperplanes $H(z_0)$, and the proposal of \cite{CH15} is to take it as the definition of the extension, with values in the dual projective space. Recall that a domain $\Omega\subset\C^n$ is \emph{weakly linearly convex} if through each point of its boundary there passes a complex hyperplane that does not meet $\Omega$ \cite{APS,Hor}. Every convex domain is weakly linearly convex, but the notion is weaker, and it is invariant under Möbius transformations. For the boundary results we require somewhat more than $\delta<1$, namely $\beta\sqrt{n+1}>1/n$, which is achieved by making $\alpha_0$ smaller.

\begin{thm}[\cite{CH15}]\label{thm:AW}
Let $F\in\mathcal F_{\alpha_0}$, with $\alpha_0$ as just indicated.
\begin{enumerate}
\item[(a)] $F(\B)$ is a weakly linearly convex domain.
\item[(b)] If in addition $F$ is normalized and quasiregular, its continuous extension to $\overline{\B}$ is injective; the hyperplanes $H(z_0)$ are disjoint from $F(\overline{\B})$; and as $z_0\to\xi\in\partial\B$, $H(z_0)$ accumulates on hyperplanes that pass through $F(\xi)$.
\item[(c)] In the situation of \textup{(b)}, the mapping
\[
E_F(z)=\begin{cases}F(z),&|z|\le1,\\ H(z^*),&|z|>1,\end{cases}\qquad z^*=z/|z|^2,
\]
is injective and continuous. It assigns distinct hyperplanes to distinct points of the exterior, and it glues $F(\B)$ to a complementary domain in the projective dual of $\C^n$.
\end{enumerate}
\end{thm}

Part (a) does not require quasiregularity, since it is obtained from (b) applied to the dilations $F_r(z)=F(rz)$ and passing to the limit. The proofs of (b) and (c) are careful estimates, based on Theorem~\ref{thm:distorsion}, of the distance between $F(z_0)$ and a distinguished point of the hyperplane $H(z_0)$, and we shall not reproduce them; see \cite{CH15}.

Questions (B) and (C) are thus answered in the ball. That the geometric conclusion is weak linear convexity, and not ordinary convexity, is consistent with everything above. The hypothesis $\|S_F\|\le\alpha_0$ is invariant under Möbius transformations, and it can imply only properties that are also invariant. What is not known is the best value of $\alpha_0$, not even its order of magnitude in $n$.

\section{The polydisk}\label{sec:polidisco}

Everything done in the ball rests on Theorem~\ref{thm:invariancia}. To see how much depends on it, we pass to the other model domain of $\C^n$. In $\Pn$ the norm $\|S_F\|$ is defined with the metric \eqref{eq:bergmanpoli}, and the rule \eqref{eq:cadenaop} remains valid; but the automorphisms $\psi_a$ of \eqref{eq:autpoli} are not Möbius, their Schwarzian is not zero, and the most that one obtains is
\begin{equation}\label{eq:psi}
\|S_{F\circ\psi_a}\|\le\|S_F\|+\|S_{\psi_a}\|\le\|S_F\|+2\sqrt2\,|a|_\infty .
\end{equation}
The class $\{\|S_F\|\le\alpha\}$ is no longer linearly invariant. What is gained in exchange is rigidity. The product structure makes the hypothesis much more restrictive than in the ball.

\begin{lem}[\cite{CH09}]\label{lem:poli}
If $F:\Pn\to\C^n$ is locally biholomorphic with $\|S_F\|\le\alpha<\infty$, then
\begin{equation}\label{eq:cotapoli}
|S^k_{ij}F(z)|\le\frac{\sqrt2\,\alpha\,(1-|z_k|^2)}{(1-|z_i|^2)(1-|z_j|^2)},\qquad z\in\Pn,
\end{equation}
and consequently $S^k_{ij}F\equiv0$ whenever $k\ne i$ and $k\ne j$.
\end{lem}

\begin{proof}
The inequality is obtained by evaluating $S_F(z)$ on the unit vectors $2^{-1/2}(1-|z_i|^2)e_i$, and polarizing when $i\ne j$. If $k\ne i,j$, the right-hand side of \eqref{eq:cotapoli} tends to zero as $|z_k|\to1$ with the other variables fixed; since $S^k_{ij}F$ is holomorphic in $z_k$, the maximum principle gives $S^k_{ij}F\equiv0$.
\end{proof}

Nothing similar occurs in the ball, where Lemma~\ref{lem:cotas} gives only bounds. The vanishing of these coefficients has structural consequences. If $F(\xi)=0$ and $DF(\xi)=\Id$, each component can be written
\begin{equation}\label{eq:estructura}
f_i(z)=\sum_{m\ge1}a_{im}(z_1,\dots,\widehat{z_i},\dots,z_n)\,(z_i-\xi_i)^m,
\end{equation}
with coefficients that do not depend on $z_i$; moreover $\partial f_i/\partial z_i\ne0$ in $\Pn$, and any two components are related by $f_i=\lambda_{ij}f_j+\mu_{ij}$ with $\lambda_{ij},\mu_{ij}$ independent of $z_j$. With this, a hypothetical failure of injectivity of $F$ is reduced, variable by variable, to a failure of injectivity of a function of \emph{one} variable, and to that function a classical criterion can be applied.

\begin{thm}[\cite{CH09}]\label{thm:poli}
Let $n\ge2$ and let $F:\Pn\to\C^n$ be locally biholomorphic. If
\[
\|S_F\|\le\frac{1}{3\sqrt2},
\]
then $F$ is univalent in $\Pn$.
\end{thm}

\begin{proof}[Outline]
Suppose $F(\xi)=F(\zeta)$ with $\xi_1\ne\zeta_1$, and normalize $F(\xi)=0$, $DF(\xi)=\Id$. Using \eqref{eq:estructura} and the relations between components, one shows inductively that $f(w)=f_1(w,\xi_2,\dots,\xi_n)$ is locally injective in $\D$ and has $f(\xi_1)=f(\zeta_1)=0$. At the points $(w,\xi_2,\dots,\xi_n)$ the differential $DF$ is diagonal and, after composing with a suitable Möbius transformation to make $\partial^2f_2/\partial z_1\partial z_2$ vanish along that line, we find
\[
\frac{f''}{f'}(w)=S^1_{11}F-2\,S^2_{12}F .
\]
By \eqref{eq:cotapoli}, $|f''/f'|\le3\sqrt2\,\alpha\,(1-|w|^2)^{-1}\le(1-|w|^2)^{-1}$. Becker's criterion says that $f$ is injective, a contradiction.
\end{proof}

That the criterion one arrives at is Becker's and not Nehari's is another manifestation of Remark~\ref{obs:n1}, since the operator $S_F$ is of second order. To measure the quality of the constant we take $F(z,w)=(f(z),w)$ in $\mathbb{D}^2$. By Example~\ref{ex:producto},
\[
\|S_F\|=\frac{\sqrt2}{3\sqrt3}\,\sup_{z\in\D}\,(1-|z|^2)\Bigl|\frac{f''}{f'}(z)\Bigr|,
\]
so that the hypothesis of the theorem is equivalent to $(1-|z|^2)|f''/f'|\le\sqrt3/2\approx0.866$, which compares reasonably with Becker's sharp constant $1$. In higher dimension, with $F(z)=(f(z_1),z_2,\dots,z_n)$, one obtains $\|S_F\|=\tfrac{1}{\sqrt2}\tfrac{n-1}{n+1}\sup(1-|z_1|^2)|f''/f'|$ for $n\ge3$, and the hypothesis is equivalent to $(1-|z|^2)|f''/f'|\le\tfrac13\tfrac{n+1}{n-1}$, which moves away from $1$ as $n$ increases. In the opposite direction, there are many univalent mappings of $\Pn$ with infinite norm. For the Roper--Suffridge extension of a univalent $f$, Example~\ref{ex:RS} gives in $\mathbb{D}^2$
\[
\|S_{\Phi_f}(z,w)\|=\frac{1}{2\sqrt2}\,\frac{|w|\,|Sf(z)|\,(1-|z|^2)^2}{1-|w|^2}\longrightarrow\infty\qquad(|w|\to1),
\]
unless $f$ is a Möbius transformation.

\medskip

\noindent\textbf{The normalized solution.} Without linear invariance there is neither a variational method nor a convex density, and in \cite{CH23} one works directly with the solutions of the system, comparing along radii with a Riccati equation.

\begin{thm}[\cite{CH23}]\label{thm:noceros}
Let \eqref{eq:sistema} be a completely integrable system in canonical form in $\Pn$ whose associated Schwarzian has norm at most $\alpha$, and let $u$ be the solution with $u(0)=1$, $\nabla u(0)=0$. If
\[
n\sqrt n\,\alpha\le\frac{1}{6.1},
\]
then $u$ does not vanish in $\Pn$, and
\[
\bigl(1-|z|_\infty^2\bigr)^{\sqrt n/2}\le|u(z)|\le\bigl(1-|z|_\infty^2\bigr)^{-\sqrt n/2}.
\]
\end{thm}

Compare with Theorem~\ref{thm:nocerosbola}. Just as in \eqref{eq:hiperplano}, the theorem says that if $F$ is normalized at the origin and $\|S_F\|\le\alpha$, with $\alpha$ in the indicated range, the image $F(\Pn)$ omits the hyperplane $1+\bil{a}{w}=0$ with $(n+1)a=\nabla J_F(0)$. It is the Ahlfors--Weill extension \emph{at one point}. What has not been achieved in the polydisk is an associated density that is at once invariant and convex with respect to the Bergman metric, and without it the argument of Section~\ref{sec:AW} cannot be repeated at all points simultaneously. The inequality \eqref{eq:psi} allows us, nevertheless, to transfer the information from the origin to an arbitrary point $z$, at the price of replacing $\alpha$ by $\alpha+2\sqrt2|z|_\infty$. In the same paper it is proved that the class of normalized mappings with $\|S_F\|\le\alpha$ is compact, that its order $\lambda_\alpha=\sup|\nabla J_F(0)|$ is finite and controls the radius of the largest Euclidean ball contained in all the images, and that $\lambda_\alpha\le2(n+1)+m(\varepsilon,n)\,\alpha^{1+\varepsilon}$ for each $\varepsilon>0$.

\section{Convex mappings}\label{sec:convexas}

In the disk, convex functions are the main example of the Nehari class. Chuaqui, Duren and Osgood \cite{CDO} proved that every convex function satisfies $\|Sf\|\le2$, with equality only for mappings onto a strip. It is natural to ask how large $\|S_F\|$ can be for a convex mapping of $\C^n$, that is, a biholomorphic mapping with convex image. The question is posed in \cite{H26}. In the polydisk the convex mappings are known explicitly, and the problem reduces to a finite-dimensional maximum which is computed when $n=2$ and estimated in general; in the ball only a family of examples is available, for which the norm is computed exactly. We follow the approach of \cite{H26}.

\medskip

\noindent\textbf{The polydisk.} By a theorem of Suffridge \cite{Suf70}, every convex biholomorphic mapping of $\Pn$ has the form
\begin{equation}\label{eq:suffridge}
F(z)=M\bigl(\varphi_1(z_1),\dots,\varphi_n(z_n)\bigr)+b,
\end{equation}
with $M$ an invertible matrix, $b\in\C^n$ and each $\varphi_k$ a convex function of the disk. The affine part does not alter the Schwarzian, so that it suffices to consider $F=(\varphi_1(z_1),\dots,\varphi_n(z_n))$. Then $J_F=\prod_k\varphi_k'(z_k)$ and, by \eqref{eq:SP},
\[
S_F(z)(v,v)=\bigl(\delta_1v_1,\dots,\delta_nv_n\bigr),\qquad\delta_k=\frac{\varphi_k''}{\varphi_k'}(z_k)\,v_k-\frac{2}{n+1}\sum_{j=1}^n\frac{\varphi_j''}{\varphi_j'}(z_j)\,v_j .
\]
To read off the Bergman norm it is convenient to normalize. Let
\[
x_k=\frac{\sqrt2\,v_k}{1-|z_k|^2},\qquad p_k=(1-|z_k|^2)\,\frac{\varphi_k''}{\varphi_k'}(z_k),
\]
so that $\|v\|_z=1$ is equivalent to $\sum|x_k|^2=1$. A direct calculation with \eqref{eq:bergmanpoli} gives
\begin{equation}\label{eq:normaconvexa}
\|S_F(z)\|^2=\frac12\,\max_{|x|=1}\ \sum_{k=1}^n|x_k|^2\,\Bigl|p_kx_k-\frac{2}{n+1}\sum_{j=1}^np_jx_j\Bigr|^2 .
\end{equation}
For a convex function we have $\bigl|(1-|z|^2)\varphi''/\varphi'-2\overline z\bigr|\le2$, hence $|p_k|<4$, and the value $4$ is attained in the limit for mappings onto a half-plane.

\begin{thm}\label{thm:convexaspoli}
Let $n\ge2$ and let $F:\Pn\to\C^n$ be a convex biholomorphic mapping.
\begin{enumerate}
\item[(a)] If $n>2$, then
\begin{equation}\label{eq:cotaconvexa}
\|S_F\|\le\frac{\sqrt{8(n+3)(n-1)}}{n+1}<2\sqrt2 .
\end{equation}
\item[(b)] If $n=2$, then $\|S_F\|\le2$, and the constant $2$ cannot be improved.
\end{enumerate}
\end{thm}

\begin{proof}
For each fixed $x$, the right-hand side of \eqref{eq:normaconvexa} is a convex function of $p=(p_1,\dots,p_n)$, since it is a sum of squares of moduli of linear forms. Its maximum over the closed polydisk $\{|p_k|\le4\}$ is therefore attained when $|p_k|=4$ for all $k$. If $p_k=4e^{i\theta_k}$, the change $x_k\mapsto e^{i\theta_k}x_k$ preserves the unit sphere and removes the phases. Therefore $\|S_F\|\le\gamma_n$, where
\begin{equation}\label{eq:Lambda}
\gamma_n=2\sqrt2\ \max_{x\in\C^n,\ |x|=1}\Bigl(\sum_{k=1}^n|x_k|^2\,\Bigl|x_k-\frac{2}{n+1}\sum_{j=1}^nx_j\Bigr|^2\Bigr)^{1/2}.
\end{equation}

For (a) it suffices to bound $\gamma_n$, and the estimate, which holds for all $n\ge2$, is that of \cite{H26}. The sum in \eqref{eq:Lambda} does not exceed the largest of the factors $\bigl|x_k-\frac{2}{n+1}\sum_jx_j\bigr|^2$, because $\sum_k|x_k|^2=1$, and by the Cauchy--Schwarz inequality
\[
\Bigl|x_k-\frac{2}{n+1}\sum_jx_j\Bigr|^2\le\Bigl(\frac{n-1}{n+1}\Bigr)^2+(n-1)\Bigl(\frac{2}{n+1}\Bigr)^2=\frac{(n-1)(n+3)}{(n+1)^2}.
\]

For $n=2$ the maximum in \eqref{eq:Lambda} can be computed exactly, and it gives a smaller constant. Write $a=|x_1|$, $b=|x_2|$, $t=ab\in[0,\tfrac12]$. Since $x_1-\tfrac23(x_1+x_2)=\tfrac13(x_1-2x_2)$,
\[
\sum_k|x_k|^2\Bigl|x_k-\tfrac23(x_1+x_2)\Bigr|^2\le\frac19\bigl[a^2(a+2b)^2+b^2(b+2a)^2\bigr]=\frac19\bigl(1+4t+6t^2\bigr)\le\frac12,
\]
with equality for $x=(1,-1)/\sqrt2$. Hence $\gamma_2=2\sqrt2\cdot\tfrac{1}{\sqrt2}=2$ and $\|S_F\|\le2$. To see that the constant cannot be improved we take $\varphi_1(\zeta)=\varphi_2(\zeta)=\zeta/(1-\zeta)$ and $z=(r,r)$ with $0<r<1$. Then $p_1=p_2=2(1+r)\to4$ as $r\to1$, and by \eqref{eq:normaconvexa} $\|S_F(z)\|\to\gamma_2=2$.
\end{proof}

The proof gives something more precise than part (a). For every $n$ we have $\|S_F\|\le\gamma_n$, with $\gamma_n$ the number defined in \eqref{eq:Lambda}, and the example at the end, with $\varphi_k(\zeta)=\zeta/(1-\zeta)$ for all $k$ and $z=(r,\dots,r)$, shows that this bound is attained in the limit. Evaluating \eqref{eq:Lambda} at $x=(1,-1,0,\dots,0)/\sqrt2$ and at $x=e_1$ we obtain in addition
\begin{equation}\label{eq:cotasLambda}
\max\Bigl\{2,\ 2\sqrt2\,\frac{n-1}{n+1}\Bigr\}\le\gamma_n\le\frac{\sqrt{8(n+3)(n-1)}}{n+1},
\end{equation}
so that the bound \eqref{eq:cotaconvexa} has the correct size and $\gamma_n\to2\sqrt2$ as $n\to\infty$. But it is not sharp. For $n=2$ it would give $2.108$, against the exact value $2$ of part (b), and in general the two inequalities used to obtain it pull in opposite directions. The first asks that $x$ be concentrated in one coordinate, and the second that it be spread over all of them. Numerically, $\gamma_3\approx2.022$, $\gamma_4\approx2.177$ and $\gamma_5\approx2.292$, against the values $2.449$, $2.592$ and $2.667$ of the bound. The maximum in \eqref{eq:Lambda} is attained at vectors with one dominant coordinate and the remaining ones equal to each other, which suggests that $\gamma_n$ is a root of an explicit polynomial. For $n\ge3$ the value of $\gamma_n$ is not known, and we know only that it lies between the bounds of \eqref{eq:cotasLambda}.

The coincidence of $\gamma_2=2$ with the constant of Chuaqui, Duren and Osgood is striking, but one should not read too much into it. By Example~\ref{ex:producto}, for these mappings $S_F$ is built from the pre-Schwarzians $\varphi_k''/\varphi_k'$ and not from the Schwarzians of the $\varphi_k$. What the theorem does say is that the convex mappings of $\Pn$ have uniformly bounded Schwarzian, with a bound that does not grow with the dimension, and that the extremal mappings are, as in the disk, mappings with unbounded image. Comparing with Theorem~\ref{thm:poli} we see also that the class of convex mappings of the polydisk does not fit in the range in which the Schwarzian detects univalence, since their norm can approach $\gamma_n\ge2$, while the criterion requires $\|S_F\|\le1/(3\sqrt2)$. In the disk, by contrast, all convex functions belong to the Nehari class.

\medskip

\noindent\textbf{The Roper--Suffridge extension in the ball.} In the ball there is no description of the convex mappings comparable to \eqref{eq:suffridge}; we denote by $K(\B)$ the class of normalized convex mappings. There are examples, however. If $\varphi$ is convex in $\D$, its Roper--Suffridge extension $\Phi_\varphi$, defined in \eqref{eq:RS}, is convex in $\B$ \cite{RS}. By Example~\ref{ex:RS},
\begin{equation}\label{eq:SRS}
S_{\Phi_\varphi}(z)(v,v)=\tfrac12\,S\varphi(z_1)\,v_1^2\,\bigl(0,z_2,\dots,z_n\bigr).
\end{equation}

\begin{thm}\label{thm:RSbola}
Let $\varphi$ be locally univalent in $\D$ and $F=\Phi_\varphi$ its Roper--Suffridge extension to $\B$, $n\ge2$. Then, for $z=(z_1,z')\in\B$,
\begin{equation}\label{eq:normaRS}
\|S_F(z)\|=\frac{(1-|z_1|^2)^{3/2}\,|z'|\,|S\varphi(z_1)|}{2\sqrt{n+1}},
\end{equation}
and therefore
\[
\|S_F\|=\frac{\|S\varphi\|}{2\sqrt{n+1}},
\]
where $\|S\varphi\|$ is the classical norm \eqref{eq:normaclasica}. In particular, if $\varphi$ is convex then $\|S_F\|\le1/\sqrt{n+1}$, with equality when $\varphi(\D)$ is a strip, and if $\varphi$ is univalent then $\|S_F\|\le3/\sqrt{n+1}$.
\end{thm}

\begin{proof}
By \eqref{eq:SRS}, $\|S_F(z)(v,v)\|_z=\tfrac12|S\varphi(z_1)|\,|v_1|^2\,\|(0,z')\|_z$. From \eqref{eq:bergmanbola},
\[
\|(0,z')\|_z^2=\frac{n+1}{(1-|z|^2)^2}\bigl[(1-|z|^2)|z'|^2+|z'|^4\bigr]=\frac{(n+1)(1-|z_1|^2)\,|z'|^2}{(1-|z|^2)^2}.
\]
It remains to maximize $|v_1|^2$ over the vectors of Bergman norm one. The matrix of the metric is $\frac{n+1}{(1-|z|^2)^2}\bigl[(1-|z|^2)\Id+\overline z\,z^t\bigr]$, its inverse is $\frac{1-|z|^2}{n+1}\bigl[\Id-\overline z\,z^t\bigr]$, and the required maximum is the first diagonal entry of the inverse, that is,
\[
\max_{\|v\|_z=1}|v_1|^2=\frac{(1-|z|^2)(1-|z_1|^2)}{n+1}.
\]
Multiplying, we obtain \eqref{eq:normaRS}. Since $|z'|^2<1-|z_1|^2$, with supremum $1-|z_1|^2$, it follows that
\[
\sup_{z'}\|S_F(z_1,z')\|=\frac{(1-|z_1|^2)^2\,|S\varphi(z_1)|}{2\sqrt{n+1}}.
\]
The rest follows from $\|S\varphi\|\le2$ for convex functions \cite{CDO} and from $\|S\varphi\|\le6$ for univalent functions.
\end{proof}

The maximum of $|v_1|$ is not attained at a vector parallel to $e_1$, except when $z_1=0$ or $z'=0$. The optimal vector has components in the direction of $z'$. This is, once again, the anisotropy of the Bergman metric.

The theorem admits three readings. The first is that the Roper--Suffridge extension carries the classical Schwarzian to the ball exactly, with the norm multiplied by $1/(2\sqrt{n+1})$; in particular $\Phi_\varphi\in\Fa$ with $\alpha=\|S\varphi\|/(2\sqrt{n+1})$, which provides the family of Section~\ref{sec:familia} with an abundant source of examples. The second is the contrast between the two domains. The same mapping, measured with the metric of $\mathbb{D}^2$, has infinite norm. The third is a lower bound for the general problem,
\[
\sup_{F\in K(\B)}\|S_F\|\ \ge\ \frac{1}{\sqrt{n+1}} .
\]
That the supremum on the left is finite is known for $n=2$, by the consequence of Theorem~\ref{thm:ordennorma} mentioned in Section~\ref{sec:familia}. Its value is an open problem.

\section{Notes and open problems}\label{sec:notas}

\noindent\textbf{Section~\ref{sec:modelo}.} The Nehari criterion is from \cite{Nehari}; the proof we gave, by Sturm comparison along a diameter, is the usual one. For the history of the Schwarzian and the theorems of Kraus, Nehari and Becker, see the book of Duren \cite{Dur83}; the derivation of Schwarz by elimination of parameters is presented in \cite[Section~10.5]{Dur04}. The extension of Ahlfors and Weill is from \cite{AW}; its interpretation in terms of the hyperbolic metric of the image and of the best Möbius approximation is due to Chuaqui and Osgood \cite{CO94}, whose distortion theorems \cite{CO93} are also the model for several estimates of this chapter. Linearly invariant families were introduced by Pommerenke in \cite{Pom}.

\medskip

\noindent\textbf{Section~\ref{sec:schwarziana}.} The definition \eqref{eq:oda} and Theorem~\ref{thm:oda} are due to Oda \cite{Oda}; the canonical form of the systems is due to Yoshida \cite{Yos76}, who studied them in connection with uniformizing equations \cite{Yos84}. The proof we present, based on the constancy of the determinant $W$, makes explicit the role of the condition \eqref{eq:canonica}. There are other notions of Schwarzian derivative in several variables, due to Molzon and Mortensen \cite{MM94,MM96} and to Molzon and Tamanoi \cite{MT}, formulated in terms of projective connections; that of Oda has the advantage of its direct link with linear equations. The operator $S_F$ and its norm were introduced in \cite{H06}; the relation with Halley's method is from \cite{H22}, and the results on the pre-Schwarzian, from \cite{H11}. The relation between the Schwarzian and the extrinsic curvature of the preimages of hyperplanes arose in joint work of the author with Osvaldo Venegas, of Universidad Católica de Temuco, which treats the case $n=2$.

\begin{prob}
The bound \eqref{eq:curvaturabergman} gives, for the mappings of $\Fa$, a control of the Euclidean curvature of the preimages of hyperplanes. Can one obtain from it, by a geometric argument, a criterion for univalence, that is, that two of these preimages in general position meet in at most one point?
\end{prob}

\noindent\textbf{Section~\ref{sec:familia}.} Lemma~\ref{lem:cotas} and Theorems~\ref{thm:normal} and~\ref{thm:ordennorma} are from \cite{H06}; the characteristic equation, the bounds for the order and the example of Theorem~\ref{thm:ordeninf}, from \cite{CH13}. The variational method is that of Pommerenke \cite{Pom}. The concepts of order and norm order in $\C^n$ are due to Pfaltzgraff \cite{Pf97} and to Pfaltzgraff and Suffridge \cite{PS}.

\begin{prob}
Determine the order and the norm order of $\Fa$ in $\B$. Between Theorems~\ref{thm:ordensup} and~\ref{thm:ordeninf} there is a gap in the dependence on $n$, and the characteristic equation \eqref{eq:carac1}--\eqref{eq:carac2} contains more information than has been used.
\end{prob}

\noindent\textbf{Sections~\ref{sec:univalencia} and~\ref{sec:AW}.} The criteria by Sturm comparison are from \cite{H07}, and the criterion by Loewner chains, from \cite{H22}; this last paper is dedicated to the memory of Gabriela Kohr. The theory of Loewner chains in $\C^n$ and Theorem~\ref{thm:pfaltzgraff} are due to Pfaltzgraff \cite{Pf74}; a complete exposition can be found in \cite{GK}. The results of Section~\ref{sec:AW} are from \cite{CH15}. The proof of univalence is that of \cite{CH15}; the control of the curvature of $\Sigma$ used there is obtained here from Theorem~\ref{thm:curvatura}. The argument of Remark~\ref{obs:univdirecta} is a shorter variant. For linear convexity and its variants, see \cite{APS,Hor}.

\begin{prob}
Determine the largest $\alpha_0(n)$ for which $\|S_F\|\le\alpha_0$ implies univalence in $\B$. Theorem~\ref{thm:univCH15} gives a value of order $n^{-3/2}$, determined by the condition $\delta(n,\alpha_0)<1$ of \eqref{eq:delta}. By Theorem~\ref{thm:RSbola}, the Roper--Suffridge extensions of nonunivalent functions with $\|S\varphi\|$ slightly larger than $2$ show that $\alpha_0(n)\le1/\sqrt{n+1}$. The correct order of magnitude is not known.
\end{prob}

\noindent\textbf{Section~\ref{sec:polidisco}.} Theorem~\ref{thm:poli} is from \cite{CH09}, and Theorem~\ref{thm:noceros}, from \cite{CH23}.

\begin{prob}
Construct in the polydisk a density associated with the solutions of \eqref{eq:sistema} that is invariant and convex with respect to the Bergman metric, or show that none exists. A complete Ahlfors--Weill extension in $\Pn$ depends on this. In the same vein, what is the optimal constant in Theorem~\ref{thm:poli}?
\end{prob}

\noindent\textbf{Section~\ref{sec:convexas}.} The theorem of Chuaqui, Duren and Osgood is from \cite{CDO}; the structure of the convex mappings of the polydisk, from Suffridge \cite{Suf70}; the extension operator, from Roper and Suffridge \cite{RS}. The problem of bounding the Schwarzian of convex mappings in several variables is posed in \cite{H26}. The order of $K(\B)$ is not known for $n\ge2$. Barnard, FitzGerald and Gong \cite{BFG} proved $\ord K(\mathbb{B}^2)\le1.761$ and conjectured that $\ord K(\B)=(n+1)/2$, which Pfaltzgraff and Suffridge disproved in \cite{PS}.

\begin{prob}
Determine $\sup\{\|S_F\|:\ F\in K(\B)\}$. By Theorem~\ref{thm:RSbola} it is at least $1/\sqrt{n+1}$. Is the Roper--Suffridge extension of the strip extremal?
\end{prob}

\begin{prob}
Find a closed expression for the constant $\gamma_n$ of \eqref{eq:Lambda} when $n\ge3$. It is the best possible bound in Theorem~\ref{thm:convexaspoli}.
\end{prob}

\section{Exercises}\label{sec:ejercicios}

\begin{enumerate}
\item Verify that for $n=1$ the formula \eqref{eq:S0} gives $S^0_{11}f=-\tfrac12Sf$, and that $u=(f')^{-1/2}(1+a_2f)$ is the solution of $u''+\tfrac12(Sf)u=0$ with $u(0)=1$, $u'(0)=0$ when $f(z)=z+a_2z^2+\cdots$.

\item Let $f$ be as in the previous exercise, with $\|Sf\|\le2$, and $u$ the normalized solution. Adapting the proof of Theorem~\ref{thm:nehari}, prove that $|u(z)|\ge\sqrt{1-|z|^2}$ in $\D$. Deduce that $-1/a_2\notin f(\D)$.

\item Let $f(z)=z+\tfrac16f'''(0)z^3+\cdots$ be analytic near the origin. Verify \eqref{eq:osculante} in this case, and deduce the general case by composing with Möbius transformations.

\item Prove the formula \eqref{eq:jacmobius} for the Jacobian of a Möbius transformation, and deduce that $\nabla J_{T_a}(0)=-(n+1)a$.

\item Show that $T_a$ is holomorphic in $\B$ if and only if $|a|\le1$, and in $\Pn$ if and only if $\sum_i|a_i|\le1$. Conclude that the order of the family $\mathcal F_0$ is $(n+1)/2$ in the ball, and that $\lambda_0=n+1$ in the polydisk.

\item Check the count \eqref{eq:recuento}, that is, show that $\tfrac12n(n-1)(n+2)+n=\tfrac12n^2(n+1)$, and that the inequality on the right is strict exactly when $n\ge2$.

\item Let $F(z_1,z_2)=(z_1+\varphi(z_2),z_2)$ with $\varphi$ holomorphic in $\D$, $\varphi(0)=\varphi'(0)=0$. Show that $F$ is biholomorphic in $\mathbb{B}^2$ with constant Jacobian, that the only nonzero coefficient of its Schwarzian is $S^1_{22}F=\varphi''(z_2)$, and that $\|S_F(0)\|=|\varphi''(0)|/\sqrt3$. Conclude that there is no bound for $\|S_F\|$ valid for all normalized biholomorphic mappings of the ball.

\item Complete the calculations of Examples~\ref{ex:producto} and~\ref{ex:RS}.

\item\label{ej:gauss} Let $\phi$ be holomorphic in a plane domain and $\Sigma=\{(\phi(w),w)\}$ its graph in $\C^2$. Show that the induced metric on $\Sigma$ is $(1+|\phi'|^2)|dw|^2$, that its Gauss curvature is $K=-2|\phi''|^2(1+|\phi'|^2)^{-3}$, and that $\kappa_\Sigma=|\phi''|(1+|\phi'|^2)^{-3/2}$ in all directions. Compare with the curvature of a plane curve.

\item Deduce \eqref{eq:cubica} from \eqref{eq:curvatura}. For $F(z_1,z_2)=(f(z_1),z_2)$ compute the cubic form $c_z$ with Example~\ref{ex:producto}, and describe the preimages of the lines that have zero curvature at $z$.

\item Check that the functions $e^{\kappa z_1}$, $e^{-n\kappa z_1}$ and $e^{\kappa z_1}z_j$ solve the system that appears in the proof of Theorem~\ref{thm:ordeninf}, and compute the maximum of $\|S(0)(v,v)\|$.

\item Verify the identity $|J_\sigma(z)|=\bigl((1-|\sigma(z)|^2)/(1-|z|^2)\bigr)^{(n+1)/2}$ for the automorphism \eqref{eq:sigma}, and complete the proof of Lemma~\ref{lem:rhoinv}.

\item Let $\rho$ be a function of class $C^2$ in $(-1,1)$ with $\rho''\ge\beta^2\rho>0$, $\rho(0)=1$ and $\rho'(0)=0$. Prove that $\rho(t)\ge\cosh(\beta t)$. This is the final step of the proof of Theorem~\ref{thm:nocerosbola}.

\item Let $F(z,w)=(f(z),w)$ in $\mathbb{D}^2$. Using Example~\ref{ex:producto}, prove that
\[
\|S_F\|=\frac{\sqrt2}{3\sqrt3}\,\sup_{z\in\D}\,(1-|z|^2)\Bigl|\frac{f''}{f'}(z)\Bigr| .
\]

\item Show that the inverse of the matrix $(1-|z|^2)\Id+\overline z\,z^t$ is $(1-|z|^2)^{-1}\bigl[\Id-\overline z\,z^t\bigr]$, and that for a positive definite Hermitian form with matrix $G$ the maximum of $|v_1|^2$ over the vectors of norm one is $(G^{-1})_{11}$. With this, complete the proof of Theorem~\ref{thm:RSbola}.

\item Verify the lower bounds of \eqref{eq:cotasLambda}, and that for $n=2$ the maximum in \eqref{eq:Lambda} is attained at $x=(1,-1)/\sqrt2$.
\end{enumerate}


\end{document}